\documentclass[hidelinks,onefignum,onetabnum]{siamart250211}

\usepackage{lipsum}
\usepackage{amsfonts}
\usepackage{graphicx}
\usepackage{epstopdf}
\usepackage{algorithmic}
\ifpdf
  \DeclareGraphicsExtensions{.eps,.pdf,.png,.jpg}
\else
  \DeclareGraphicsExtensions{.eps}
\fi

\newsiamremark{remark}{Remark}
\newsiamremark{hypothesis}{Hypothesis}
\crefname{hypothesis}{Hypothesis}{Hypotheses}
\newsiamthm{claim}{Claim}
\newsiamremark{fact}{Fact}
\crefname{fact}{Fact}{Facts}

\headers{Subsample Size of Quantile Randomized Kaczmarz}{J.-F. Cai, J. Chen, A. Ma, and T. Wu}

\title{On the Subsample Size of Quantile-Based Randomized Kaczmarz
\thanks{Received by the editors August 6, 2025; accepted for publication (in revised form) by U. M. Yang February 23, 2026; published electronically June 19, 2026.
\funding{The first author was supported in part by the Hong Kong Research Grant Council (HKRGC) GRFs 16307023, 16306124, and 16307325. The second author was supported in part by a Novikov Postdoctoral Fellowship from the Department of Mathematics, University of Maryland, College Park; the third author was supported in part by the Rose Hills Innovators Fellowship and the Simons Grant MPS-TSM-00014026; the fourth author was supported in part by the Hong Kong PhD Fellowship Scheme (HKPFS).
 }
}
}

\author{Jian-Feng Cai\thanks{Department of Mathematics,
Hong Kong University of Science and Technology, Hong Kong. {\it email:} \url{jfcai@ust.hk}}\and Junren Chen\thanks{Department of Mathematics,
  University of Maryland, College Park, MD 20742-4015 USA. {\it email:} \url{jchen58@umd.edu}} \and Anna Ma\thanks{Department of Mathematics,
University of California, Irvine, CA 92697 USA. {\it email:} \url{anna.ma@uci.edu}}\and Tong Wu\thanks{Department of Mathematics,
Hong Kong University of Science and Technology, Hong Kong. {\it email:} \url{twubi@connect.ust.hk}}}

\usepackage{amsopn}

\newtheorem*{question*}{Question}
\newtheorem*{theoreminformal*}{Theorem (Informal)}
\usepackage{bm,xcolor,url}
\usepackage[caption=false]{subfig} 
\usepackage{amssymb,amsmath,latexsym}
\usepackage{verbatim}
\usepackage{bm}
 \usepackage[nameinlink,capitalize]{cleveref}
 \usepackage{float}

 \usepackage{tabularx}
 \newcolumntype{L}{>{\raggedright\arraybackslash}X}%
 \newcolumntype{R}{>{\raggedleft\arraybackslash}X}%
 \newcolumntype{C}{>{\centering\arraybackslash}X}%
 \usepackage{hyperref}
\usepackage{relsize}
\usepackage{subfig}
\usepackage{textcomp}
\usepackage{mathrsfs}
\usepackage{ntheorem}

\newcommand{\calF}{\mathcal{F}}

\newcommand{\ba}{\mathbf{a}}
\newcommand{\bA}{\mathbf{A}}
\newcommand{\bb}{\mathbf{b}}

\newcommand{\be}{\mathbf{e}}

\newcommand{\bI}{\mathbf{I}}

\newcommand{\bu}{\mathbf{u}}

\newcommand{\bx}{\mathbf{x}}
\newcommand{\bX}{\mathbf{X}}

\newcommand{\bepsilon}{\bm{\epsilon}}

\newcommand{\D}{\operatorname{D}_{KL}}

\renewcommand{\bm}{\mathbf}

\renewcommand{\boldsymbol}{\mathbf}
\newtheorem{assumption}{Assumption}

\newcommand{\target}{\bm{x}^*}
\newcommand{\init}{\bm{x}_{0}}
\newcommand{\RHS}{\bm{b}}

\newcommand{\Prob}[3][m]{#2\exp(-#3 #1)}
\newcommand{\Cforsmallest}[2][n]{ C_1 \frac{1}{#2} \log (\frac{\sqrt{2/\pi}}{#2})#1}
\newcommand{\Pforsmallest}[1]{\Prob{3}{c_2#1}}

\newcommand{\Pforabs}{ 2\exp(-cm)}

\crefname{equation}{}{}

\newcommand{\updateid}[1]{r_{#1}}
\newcommand{\smin}[1]{\sigma_{\min,#1}}

\crefname{algocf}{Algorithm}{Algorithms}

\newcommand{\e}[1]{\epsilon_{#1}}

\renewcommand{\ln}{\log}
\newcommand{\subsample}[2]{\{#1_j^{(#2)}\}_{j=1}^{D}}

\newcommand{\Uniform}{\operatorname{Uniform}(1, \ldots, m)}
\newcommand{\Bin}[2]{{\rm Bin}(#1,#2)}
\newcommand{\note}[1]{\mbox{\footnotesize (#1)}}

\newif\ifshow 
   \newif\ifshowwt
\showwttrue

\newcommand{\wt}[2][f]{%
  \ifshowwt
    \ifnum\pdfstrcmp{#1}{i}=0%
      \textcolor{purple}{wt: #2}
    \else%
      \footnote{ \textcolor{purple}{wt:#2}
      }
    \fi%
  \else%
      \fi%
}

\showtrue  
\newcommand{\jr}[2][f]{%
  \ifshow
    \ifnum\pdfstrcmp{#1}{i}=0%
      \textcolor{blue}{#2}
    \else%
      \footnote{
      \textcolor{blue}{#2}
      }
    \fi%
  \else%
      \fi%
}

\ifpdf
\hypersetup{
  pdftitle={Subsample Size of Quantile Randomized Kaczmarz},
  pdfauthor={J.-F. Cai, J. Chen, A. Ma, and T. Wu}
}
\fi

\begin{document}

\maketitle

\begin{abstract}
Quantile-based randomized Kaczmarz (QRK) was recently introduced to efficiently solve sparsely corrupted linear systems $\mathbf{A} \mathbf{x}^* + \mathbf{\epsilon} = \mathbf{b}$ [{\em SIAM J. Matrix Anal. Appl.}, 43 (2022), pp. 605--637], where $\mathbf{A} \in \mathbb{R}^{m \times n}$ and $\mathbf{\epsilon}$ is an arbitrary $(\beta m)$-sparse corruption. However, all existing theoretical guarantees for QRK require quantiles to be computed using all $m$ samples (or a subsample of the same order), thus negating the computational advantage of Kaczmarz-type methods. This paper overcomes the bottleneck. We analyze a subsampling QRK, which computes quantiles from $D$ uniformly chosen samples at each iteration. Under some standard scaling assumptions on the coefficient matrix, we show that QRK with subsample size $D \geq \frac{C \log (T)}{\log(1/\beta)}$ linearly converges over the first $T$ iterations with high probability, where $C$ is some absolute constant. This subsample size is a substantial reduction from $O(m)$ in prior results. For instance, it translates into $O(\log(n))$ even if an approximation error of $\exp(-n^2)$ is desired.  
Intriguingly, our subsample size is also tight up to a multiplicative constant: if $D \leq \frac{c \log(T)}{\log(1/\beta)}$ for some constant $c$, the error of the $T$th iterate could be arbitrarily large with high probability. Numerical results are provided to corroborate our theory.
\end{abstract}



\section{Introduction}\label{sec11}
Solving large-scale systems of linear equations is a fundamental task with widespread applications in scientific computing, data science, engineering, and related fields. 
In practice, the right-hand side of the linear system is often corrupted due to sensor failures, transmission errors, or adversarial attacks.
As a consequence, efficient algorithms for solving corrupted linear systems have been highly sought after by applied mathematicians. 
This problem can be formulated as finding $\bx^*$ from  
\begin{equation} \label{eq: problem}
      \bA\mathbf{x}^* + \bm{\epsilon} =  \mathbf{b},
\end{equation}
where $\bA = [\ba_1, \cdots, \ba_m]^\top \in \mathbb{R}^{m \times n}$ denotes the coefficient matrix with rows $\ba_j \in \mathbb{R}^n$, $\mathbf{x}^* \in \mathbb{R}^n$ denotes the desired solution, $\mathbf{b}=[b_1, ..., b_m]^\top$ is the measurement vector that is corrupted by $(\beta m)$-sparse $\bm{\epsilon}\in \mathbb{R}^m$ whose nonzero entries can take {\it arbitrary} values, and $\beta \in (0,1)$ is small to ensure the well-posedness of the problem.

{Storing the entire linear system may not be possible in the highly overdetermined case with an extremely large $m$.} 
Kaczmarz methods have proven particularly advantageous in this regime as they only {require} one row per iteration. 
In particular, the Kaczmarz algorithm iteratively projects the iterate onto the solution hyperplane of a chosen row: suppose $\ba_j$ is chosen for updating the current iterate $\bx_k$, then one simply projects $\bx_k$ to $\{\bu\in \mathbb{R}^n:\ba_j^\top\bu=b_j\}$, {i.e.}, 
\begin{align}
    \bx_{k+1} = \bx_k - \frac{(\ba_j^\top\bx_k-b_j)\ba_j}{\|\ba_j\|^2}.  
    \label{eq:rk}
\end{align}
{The rate of convergence of the Kaczmarz algorithm heavily relies on the ordering of the rows that the iterates are projected onto \cite{herman1993algebraic}. As a consequence, only convergence to the solution (for consistent systems) was guaranteed in its initial study~\cite{karczmarz1937angenaherte}.}

The randomized Kaczmarz (RK) method was introduced by Strohmer and Vershynin, who proved that RK converges linearly in expectation \cite{strohmer2009randomized} when the linear system is consistent and rows are selected randomly with probability proportional to each row norm. 
Unfortunately, when the system is noisy and inconsistent, the best one can hope for is convergence in expectation to a ball containing the least squares solution, where the radius of the ball depends on the norm of the noise \cite{needell2010randomized,zouzias2013randomized,wang2015randomized}. When the noise is very large, this convergence error horizon grows with it, even if only one row is corrupted.

Recently, Haddock et al.~\cite{haddock_quantilebased_2022} developed a quantile-based randomized Kaczmarz (QRK) method that can solve sparsely corrupted linear systems. The core idea is to avoid projections onto corrupted rows by using quantile statistics. To update $\bx_k$ via the $j$th row, the quantile of the full-sample residuals $\{|\ba_i^\top\bx_k-b_i|\}_{i=1}^m$, or the quantile of a subsample of $\{|\ba_i^\top\bx_k-b_i|\}_{i=1}^m$, is computed, and the projection is only made when $|\ba_{j}^\top\bx_k-b_{j}|$ is below the quantile.    
{To make this precise while simplifying notation}, we assume the rows of $\bA$ have unit Euclidean norm and introduce the QRK algorithm from \cite{haddock_quantilebased_2022} {in \cref{algo:quantileRK}}. 
In~\cref{algo:quantileRK}, {$\subsample{i}{k+1}$} is the set of $D$ uniformly chosen row indices for computing the quantile (referred to as {\it quantile subsample}) and $r_{k+1}$ is the index of the row for potential projection (referred to as {\it update sample}).
\begin{algorithm}
\caption{QuantileRK($q$) \cite{haddock_quantilebased_2022}\label{algo:quantileRK}}
\begin{algorithmic}[1]
\REQUIRE$\mathbf{A}$, $\mathbf{b}$, $\mathbf{x}_0$, $q\in(0,1)$, quantile subsample size $D$, iteration number $T$
\FOR{$k = 0$ \TO $T-1$}
  \STATE Sample $i_1^{(k+1)}, \ldots, i_D^{(k+1)} \sim \operatorname{Uniform}(1, \ldots, m)$
  \STATE Sample $\updateid{k+1} \sim \operatorname{Uniform}(1, \ldots, m)$
  \STATE Compute $h_{r_{k+1}} := \ba_{r_{k+1}}^\top\bx_k - b_{r_{k+1}}$
  \IF{$|h_{r_{k+1}}| \leq q\text{-}\mathrm{quantile}\Big(\big\{|\ba_{i^{(k+1)}_j}^\top\bx_k - b_{i^{(k+1)}_j}|\big\}_{j=1}^D\Big)$}
    \STATE $\bx_{k+1} \gets \bx_{k} - h_{r_{k+1}}\ba_{\updateid{k+1}}$
  \ELSE
    \STATE $\bx_{k+1} \gets \bx_{k}$
  \ENDIF
\ENDFOR
\RETURN $\bx_T$
\end{algorithmic}
\end{algorithm}

Under a class of random coefficient matrices (including $\bA$ with i.i.d. rows uniformly drawn from the unit sphere), the authors of \cite{haddock_quantilebased_2022} established linear convergence {to the solution $\bx^*$ of the corrupted system} for QRK, similar to \cite{strohmer2009randomized}. Later, Steinerberger \cite{steinerberger2023quantile} provided an analysis over deterministic $\bA$. However, both of these guarantees are only valid for QRK using quantiles computed from {full samples}, meaning that the quantile over a random subsample of size $D$ in   \cref{algo:quantileRK} is replaced by the quantile of $\{|\ba_i^\top\bx_k-b_i|\}_{i=1}^m$. Consequently, existing {theoretical guarantees for} QRK need to access the entire {matrix $\bA$} at each iteration. This negates the major computational advantage of Kaczmarz methods, i.e., requiring only one row per iteration, and a small part of the system overall.

On the other hand, \cref{algo:quantileRK} with $D \ll m$ works well in numerical simulations (e.g., see~\cite{haddock_quantilebased_2022}). To the best of our knowledge, providing a theoretical guarantee for this small quantile subsample regime remains an important open question. Indeed, there has only been one work~\cite{haddock_subsampled_2023} that analyzed a subsampled QRK and showed that a subsample size $\alpha m$ for some $\alpha\ge \max\{\frac{\beta}{q},\frac{\beta}{1-q}\}$ suffices. 
There is an important difference between the algorithm in~\cite{haddock_subsampled_2023} and \cref{algo:quantileRK}: their algorithm selects the update sample from the quantile subsample of size $D$, in stark contrast to \cref{algo:quantileRK} that selects the update sample from all $m$ rows.   
Note that to establish QRK's linear convergence, the update sample should be drawn from an ensemble with a number of rows exceeding $\Theta(m)$ so that the standard Strohmer--Vershynin bound~\cite{strohmer2009randomized} can be applied when updating with an uncorrupted row. As a result, the QRK in~\cite{haddock_subsampled_2023} necessarily requires a quantile subsample of size $\Theta(m)$.  
However, accessing a constant fraction of the system per iteration remains computationally demanding. 

The goal of this paper is to identify the subsample size 
required for \cref{algo:quantileRK} to succeed. Unlike previously established guarantees~\cite{strohmer2009randomized,haddock_quantilebased_2022,steinerberger2023quantile,haddock_subsampled_2023}, which allow for infinite iterations, we restrict our attention to a given finite number of iterations (say, $T$). This is somewhat necessary: a bad step that projects the iterate onto a corrupted row will occur if sufficiently many iterations are run, and one such bad step can make the iterate arbitrarily distant from $\bx^*$. While this might appear to be a limitation, in practice, one always runs a finite number of iterations to achieve a desired approximation error. 
We now formally introduce the question that we focus on.

\begin{question*}
    Given a positive integer $T$, what is the size of the quantile subsample needed for \cref{algo:quantileRK} to linearly converge over the first $T$ iterations?   
\end{question*}

We denote the size of the quantile subsample (also called subsample size) by $D$ throughout this paper. Our main contributions are to provide upper and lower bounds, which match up to a multiplicative constant, for $D$. For simplicity, we assume that the rows of $\bA$ are uniformly distributed over the unit Euclidean sphere. \footnote{As we shall see, our results only require certain properties on $\bA$ that are satisfied with high probability by the class of random matrices considered in \cite[Assumptions 1--2]{haddock_quantilebased_2022}; see our   \cref{assum1}--\cref{assump2}. Moreover, it is not hard to adapt our results to a deterministic manner, as done in \cite{steinerberger2023quantile}.}
We first provide an upper bound showing that a uniformly chosen quantile subsample of size $D=O\big(\frac{\log T}{\log(1/\beta)}\big)$ is sufficient to guarantee linear convergence. Let us provide an informal statement here; see \cref{thm: main Gaussian} for the complete statement. We generically use $o(1)$ to denote quantities that tend to $0$ when $m,n,T\to\infty$.
\begin{theoreminformal*}
   There exist some constants $C_1,c_2$ such that for any $T\in\mathbb{Z}_+$ and for arbitrary $(\beta m)$-sparse $\bepsilon$,   \cref{algo:quantileRK} with $D\ge \frac{C_1\log T}{\log(1/\beta)}$ satisfies 
   $
    \|\bx_T - \bx^*\|^2\le \big(1-\frac{c_2}{n}\big)^T \|\bx_0-\bx^*\|^2$ with $1-o(1)$ probability. 
\end{theoreminformal*}

In the most interesting regime where $\beta$ is some positive constant, we need a quantile subsample of size $O(\log T)$ to run $T$ iterations. As a consequence, to achieve a squared $\ell_2$ approximation error of $\varepsilon\|\bx_0-\bx^*\|^2$, we can set $T=\frac{n\log(1/\varepsilon)}{c_2}$ and hence require a subsample size of $$D=O\big(\log n +\log\log(1/\varepsilon)\big).$$ 
Such a subsample size reduces to $D=O(\log n)$, which is substantially smaller than $\Theta(m)$ in~\cite{haddock_subsampled_2023} when $m = \Omega(n)$, even under $\varepsilon=\exp(-n^2)$—an approximation error that is more than sufficient for most practical applications.\footnote{In fact, our $D=O(\log n +\log\log(1/\varepsilon))$ improves on $\Theta(m)$ so long as $\varepsilon\ge \exp(-\exp(m))$.}

It is interesting to note that the subsample size can be further reduced if $\beta = o(1)$. Let us consider a sublinear number of corruptions: $\beta = \Theta(m^{-\xi})$ for some $\xi\in (0,1]$. To achieve a squared $\ell_2$ approximation error of $\varepsilon\|\bx_0-\bx^*\|^2$, we only require a subsample size of $D=O(\frac{\log n+\log\log\varepsilon^{-1}}{\xi\log m})$. If additionally $\varepsilon \ge \exp(-n^2)$, then we only require $D=O(\frac{\log n}{\xi\log m})=O(\xi^{-1})$, meaning that a constant subsample size (depending only on $\xi$) suffices under any $\xi\in(0,1]$.

We also provide a converse result indicating that a subsample size at the order of $\frac{\log T}{\log(1/\beta)}$ is necessary for   \cref{algo:quantileRK} to converge linearly over the first $T$ iterations. The following is an informal version of   \cref{thm:lower}.
\begin{theoreminformal*}
    Under some mild scaling assumptions on $(\beta,T)$, there exists an absolute constant $c_1$ such that for some $(\beta m)$-sparse corruption $\bepsilon$, \cref{algo:quantileRK} with $D\le \max\big\{\frac{c_1\log T}{\log(1/\beta)},1\big\}$ returns $\bx_T$ with arbitrarily large $\|\bx_T-\bx^*\|^2$ with $1-o(1)$ probability. 
\end{theoreminformal*}

Our proof techniques for the upper bound depart from existing analysis \cite{haddock_quantilebased_2022,steinerberger2023quantile,haddock_subsampled_2023} in many ways.   In particular, we first develop two-sided probabilistic bounds on the subsampled quantile. We then rely on the upper bound to control the impact of accepting a corrupted row, and utilize the lower bound to establish the contraction after accepting an uncorrupted row. Both of these steps require a number of new ideas, especially in iterating an appropriate one-step contraction and constructing a sufficiently large set of acceptable uncorrupted rows. See \cref{sec:newideas} for details. On the other hand, we build the lower bound by ensuring the projection   onto some corrupted row in some step and then showing the subsequent iterations cannot reduce the approximation error too much.


This paper is organized as follows.  We introduce the notation and some useful technical tools in  
 \cref{sec:prelim}. We present our main theoretical results and their proofs in   \cref{sec:main}. We also provide a set of numerical examples in \cref{sec:numerics} which support our theoretical findings. Last, we provide concluding remarks in \cref{sec:conclusion}.

\section{Preliminaries}

\label{sec:prelim}
    In this section, we present notation that will be used throughout for easy referencing. In addition to notation, we present key technical tools and summarize preliminary observations from prior works when \cref{algo:quantileRK} chooses a corrupted row in \cref{sec: Projection onto Corrupted Rows} and a uncorrupted row in \cref{sec:  Projection onto Uncorrupted Rows}. Some properties of the coefficient matrix $\bA$ are discussed in \cref{matrixA}.

    We adopt the convention $[m] = \{1, \dots, m\}$ and for simplicity, we assume that the proportions of $m$ (such as $qm$, $\beta m$, $\alpha m$, where $\alpha, \beta, q \in [0,1]$) and $qD$ are positive integers. If not, rounding can be applied without significantly affecting any results. 
    The inner product between $\bm{a},\bm{b}\in \mathbb{R}^n$ is defined as $\langle \bm{a}, \bm{b} \rangle = \bm{a}^\top \bm{b}$, and the norm $\|\ba\| := \sqrt{\langle \bm{a}, \bm{a} \rangle}$.  
    We work with natural logarithm $\log(\cdot)$ with base $e$. 
    The Kullback--Leibler (KL) divergence between two Bernoulli distributions with parameters $p,q\in[0,1]$, is defined as $${\rm D_{KL}}(p\|q)= p\log\Big(\frac{p}{q}\Big)+(1-p)\log\Big(\frac{1-p}{1-q}\Big).$$ 
    We denote constants by $C,C_i,c,c_i$ whose values may vary from line to line. They are either absolute constants or constants depending on certain parameters (which will be specified). We will adopt standard complexity notation, by writing $I_1=O(I_2)$ if $I_1\le CI_2$ for some   constant $C$, $I_1=\Omega(I_2)$ if $I_1\ge cI_2$ for some constant $c$, and $I_1=\Theta(I_2)$ if $I_1=O(I_2)$ and $I_1=\Omega(I_2)$ simultaneously hold.

We assume the rows of $\bm{A} \in \mathbb{R}^{m \times n}$ are independent random vectors uniformly distributed on the unit sphere $\mathbb{S}^{n-1}$. The Frobenius norm of $\bm{A}$ is denoted by $\|\bm{A}\|_F$. The smallest nonzero singular value and the largest singular value of $\bm{A}$ are denoted by $\sigma_{\min}(\bm{A})$ and $\sigma_{\max}(\bm{A})$, respectively. For any $S \subset [m]$, $\bm{A}_S \in \mathbb{R}^{|S| \times n}$ denotes the submatrix of $\bm{A}$ with rows indexed by $S$. Note that $|S|$ denotes the cardinality of a finite set $S$.

The $q$-quantile of a multiset $\{z_1,z_2, \dots, z_N\}$ (in which $z_i=z_j$ is possible for $i\ne j$) is defined as $z^*_{\lfloor qN\rfloor}$, where we let $z^*_1\le z_2^*\le\cdots \le z_N^*$ be the nondecreasing rearrangement of the $N$ elements. (If $qN<1$, we will instead define the $q$-quantile as $z_1^*$.) 
Let  $B := \{ i: \epsilon_i \neq 0 \}$ be the index set of the corrupted measurements in the right-hand side $\RHS$, and let $B^c:=[m]\setminus B$ be its complement.
At runtime, we can only use the possibly corrupted $\RHS$ to compute the quantile, where the quantile is denoted by $$Q_q(\bm{x}, S) := \text{$q$-quantile of } \left\{ \left| b_i - \langle \bm{a}_i, \bm{x} \rangle \right| : i \in S \right\},$$ 
for a given multiset $S$.
We shall denote the full-sample quantile by $$Q_q(\bm{x}) := Q_q(\bm{x}, [m]).$$ 
If $b_i$ is uncorrupted, then $b_i=\langle\ba_i,\bx^*\rangle$ and hence the absolute residual for the $i$th row can be expressed as $|\langle\ba_i,\bx-\bx^*\rangle|$. Hence, quantiles computed from uncorrupted measurements can be formulated as 
$$\tilde{Q}_q(\bm{x}, S) := \text{$q$-quantile of } \left\{ \left| \langle \bm{x} - \bm{x}^*, \bm{a}_i \rangle \right| : i \in S \right\}.$$  
Similarly, the full-sample uncorrupted quantile will be written as $$\tilde{Q}_q(\bm{x}) := \tilde{Q}_q(\bm{x}, [m]).$$
We denote by $\bX_k$ the QRK iterate after $k$ steps when we want to emphasize its randomness, and may also write it as $\bx_k$ when working with a deterministic realization of $\bX_k$. 
The error vector at iteration $k$ is defined as $\bm{e}_k := \bm{x}_k - \bm{x}^*$. 

We now proceed to introduce the technical tools used in our paper. At the center of our analysis are various binomial variables; hence, our major technical tool is the following tight Chernoff bound. We emphasize that the Chernoff bound provided below is tighter than the simpler versions used in \cite{haddock_quantilebased_2022,steinerberger2023quantile}, and such tightness is essential for capturing the role of $\beta$ in the subsample size. 

\begin{lemma}\label{thm: chernoff}
\textup{(Concentration of Binomial Random Variables; see, e.g., \cite[Sec.~4.7]{ash2012information}, \cite{arratia1989tutorial})}
Let $X\sim \Bin{N}{q}$, then we have the Chernoff bound: for $k \leq Nq$,
    \[
        \mathbb{P}\left(X \leq k\right) \leq \exp\left(-N \cdot \D\left(\frac{k}{N} \| q\right)\right).
    \]
    Similarly, for $k \geq Nq$,
    $
        \mathbb{P}(X \geq k) \leq \exp\left(-N \cdot \D\left(\frac{k}{N} \| q\right)\right).
    $
\end{lemma}

\subsection{Projection onto Corrupted Rows}
\label{sec: Projection onto Corrupted Rows}
In \cref{algo:quantileRK}, it is possible that a corrupted row is selected for updating the iterate. Fortunately, since the subsampled quantile bounds the residual, the resulting error increase can be effectively controlled, as shown in \cref{lem: quantile bound error}.
\begin{lemma} \label{lem: quantile bound error}
Given $\bx_k$, the update sample $\updateid{k+1}=i$, and the \\quantile subsample $\subsample{i}{k+1}$. If $Q(\bx_{k},\subsample{i}{k+1}) \leq m_Q$, then
\begin{equation}
 \left\|\be_{k+1}\right\|^2 \leq \left\|\bm{e}_k\right\|^2 + m_Q^2 + 2 m_Q \left|\left\langle \bm{e}_k, \bm{a}_{i} \right\rangle\right|.  
\label{eq:quantile-error-bound}
\end{equation}
\end{lemma}
\begin{proof}


In \cref{algo:quantileRK}, we can bound the residual of an accepted update sample $r_{k+1}=i$ as $\left|h_{i}\right| = \left|\left\langle \bx_k, \ba_{i} \right\rangle - b_{i}\right| \leq Q(\bm{x}_k,\subsample{i}{k+1}) \leq m_Q$. Thus, we have  
        $$ 
            \begin{aligned}
                 \left\|\be_{k+1}\right\|^2  
                 & \begin{cases}
                      =  \left\|\be_{k} - h_i \ba_i\right\|^2 &\text{if accepted} \\
                       =  \left\|\bm{e}_k\right\|^2                                & \text{if rejected}
                   \end{cases} \\
                      &\ \, \, \leq
                \left\|\bm{e}_k\right\|^2 + Q\left(\bm{x}_k,\subsample{i}{k+1}\right)^2 + 2 Q\left(\bm{x}_k,\subsample{i}{k+1}\right) \left|\left\langle\bm{e}_k, \bm{a}_{i}\right\rangle\right|
                       \\
                 & \ \, \, \leq \left\|\bm{e}_k\right\|^2 + m_Q^2 + 2 m_Q \left|\left\langle\bm{e}_k, \bm{a}_{i}\right\rangle\right|.
            \end{aligned}
        $$
        The proof is complete.
\end{proof}
\cref{lem: full quantile} relates the realized quantiles $Q_q$ with the uncorrupted quantiles $\tilde{Q}_q$ in a full sample regime. (We will derive an analogous lemma for the subsampled quantiles in \cref{sec:main}.)

\begin{lemma}[{\cite[Lemma 3.6]{haddock_quantilebased_2022}}] \label{lem: full quantile}
    Let $\bx_k$ be a fixed vector. Under  $(\beta m)$-sparse corruption, we have 
    $$
        \tilde{Q}_{q-\beta}\left(\bm{x}_k\right) \leq  \tilde{Q}_{\frac{q-\beta}{1 - \beta}}\left(\bm{x}_k, B^c\right) \leq Q_q\left(\bm{x}_k\right) \leq \tilde{Q}_{q+\beta}\left(\bm{x}_k\right).
    $$
    \end{lemma}
While the uncorrupted quantiles are not accessible during runtime, they are more technically amenable  and can be bounded by \cref{lem: A and ideal quantile}. 
\begin{lemma}[{\cite[Lemma 1]{steinerberger2023quantile}}] \label{lem: A and ideal quantile}
    For any $q' \in (0,1)$, the uncorrupted quantile satisfies:
    \begin{equation} \label{eq:ideal-quantile-upper} 
        \tilde{Q}_{q'}(\bm{x}_k) \leq \Phi_{q'} \frac{\left\|\bm{x}_k - \target\right\|}{\sqrt{n}}
        \quad \text{where}~~\Phi_{q'} := \frac{\sigma_{\max}(\bm{A}) \sqrt{n}}{\sqrt{m} \sqrt{1 - q'}}.
    \end{equation}
\end{lemma}
\cref{lem: A and ideal quantile} is a straightforward consequence of the following lemma.   
\begin{lemma}[{\cite[Lemma 3.7]{haddock_quantilebased_2022}}] 
    Let $\mathbf{A} \in \mathbb{R}^{m \times n}$. For all unit vectors $\mathbf{x} \in \mathbb{R}^n$ and every subset $S \subseteq [m]$, the following holds: 
        \begin{align}
          \sum_{i \in S}\left|\left\langle\mathbf{x}, \mathbf{a}_i\right\rangle\right| 
        \leq \sqrt{|S|} \cdot \left(\sum_{i \in S}\left|\left\langle\mathbf{x}, \mathbf{a}_i\right\rangle\right|^2\right)^{1/2} 
        \leq \sigma_{\max}(\bm{A}) \sqrt{\frac{n}{m}} \cdot \sqrt{\frac{m|S|}{n}}.
    \end{align}
    \label{lem: absolute sum bound}
\end{lemma}

\subsection{ Projection onto Uncorrupted Rows}
\label{sec:  Projection onto Uncorrupted Rows}
In this scenario, the error is nonexpansive since
\begin{equation} \label{eq: nonexpansive}
\|\bX_{k+1} - \target\|^2 = \left\| \left(\bI - \frac{\ba_{\updateid{k+1}}\ba_{\updateid{k+1}}^{\top}}{\|\ba_{\updateid{k+1}}\|^2}\right) \bm{e}_k \right\|^2 \leq \|\bm{e}_k\|^2.
\end{equation}
In expectation, one actually has an error contraction captured by the following \\
Strohmer--Vershynin bound. 
\begin{lemma}[\cite{strohmer2009randomized}]
Let $\bm{A} \bx = \bb$ be a consistent linear system. The iterates $\bX_{k+1}$ generated by RK~\eqref{eq:rk} satisfy
\begin{equation} \label{eq: contraction-noiseless2}
    \mathbb{E}\|\bX_{k+1} - \target\|^2
    \leq \left(1 - \frac{\sigma_{\min}^2(\bm{A})}{\|\bm{A}\|_F^2}\right) \|\bx_k-\bx\|^2,
\end{equation}
where $\ba_{\updateid{k+1}}$ is selected with probability proportional to $\|\ba_{\updateid{k+1}}\|^2$.
\label{Stbound}
\end{lemma}
In the analysis of QRK, the quantity
\begin{equation} \label{eq:uniform-smin}
    \smin{\alpha} := \inf _{\substack{S \subseteq[m]:|S| \geq \alpha m}} \sigma_{\min }\left(\mathbf{A}_S\right)  \sqrt{\frac{n}{m}}\quad \text{for  } \alpha \in (0,1),
\end{equation}
plays an important role in our application of \cref{Stbound}. 
The reason is that accepting an uncorrupted update sample in QRK can be treated as running one iteration of standard RK for an uncorrupted linear subsystem with more than $\alpha m$ rows.

\subsection{Coefficient Matrix $\bA$}\label{matrixA}
When $\bA$ consists of rows sampled i.i.d. uniformly from $\mathbb{S}^{n-1}$, the constants $\smin{\alpha}$ in \cref{eq:uniform-smin} and $\Phi_{q'}$ in \cref{eq:ideal-quantile-upper} are bounded with high probability when $m=\Omega(n)$. 
\begin{lemma}[{\cite[Lemma 3.7 \& Proposition 3.4]{haddock_quantilebased_2022}}] \label{lem: randomness-Gaussian}
Let $\alpha_0 \in (0,1)$. There exists absolute constants $C_1, c_2, c, C$ such that if $
m \geq \max\big\{\Cforsmallest{\alpha_0}, n\big\},
$   then with probability at least $1 - \Pforsmallest{\alpha_0} - \Pforabs$,
we have: 
\begin{equation} \label{eq: Gaussian}
    \smin{\alpha_0} \geq \frac{\sqrt{2\pi}\alpha_0^{3 / 2}}{24}, \quad \frac{\sigma_{\max}(\bm{A}) \sqrt{n}}{\sqrt{m}} \leq C, \quad  \Phi_{q'} \leq \frac{C}{\sqrt{1-q'}}\quad  (\forall q' \in (0,1)).  
\end{equation}

\end{lemma}

\section{Main Results}\label{sec:main}
We present our main theoretical results. First, we introduce an upper bound on the subsample size in \cref{sec:upper_bound} and present its proof in \cref{sec: Proof of the main}. Then, to complement the tightness of our upper bound, we provide the lower bound in \cref{sec: lower bound} and prove it in   \cref{sec:lowerproof}.

\subsection{Upper Bound} \label{sec:upper_bound}
In our upper bound, we make the following assumption on $\bA$. 
\begin{assumption}\label{assum1}
    $\bA$ has rows in $\mathbb{S}^{n-1}$ and satisfies all the spectral properties in~\cref{eq: Gaussian} with 
$\alpha_0 = \frac{q}{2}$, where $q$ is the quantile parameter in \cref{algo:quantileRK}. 
\end{assumption}

When $\bA$ has i.i.d. rows uniformly sampled from $\mathbb{S}^{n-1}$,
 these properties hold with high probability under $m \geq C_q n$, where $C_q$ is some constant depending only on $q$; see \cref{lem: randomness-Gaussian}.

We now formally present the upper bound on the subsample size $D$.

\begin{theorem} \label{thm: main Gaussian}
 Consider problem \cref{eq: problem} with arbitrary $(\beta m)$-sparse $\bm{\epsilon}$. 
 Let $q \in (0,1)$, $T>1$ be the number of iterations run, and $D$ be the subsample size used in  \cref{algo:quantileRK}. Suppose \cref{assum1} holds. 
 There exist constants $c_1,C_2,c_2,c_3,c_4$ depending only on $q$, such that for any $\beta \in (0, c_1)$, if the positive integer $D$ satisfies 
    $$D\ge \frac{C_2\log T}{\log(1/\beta)},$$ then  we have the following  guarantees on the iterate $\bX_{T}$ of \cref{algo:quantileRK}:
    \begin{itemize}
        \item (Convergence in expectation)
                 There exists a failure event $\Omega$ 
        satisfying $$\mathbb{P}(\Omega)\le T^{-5}\exp\Big(-\frac{1-q}{4}\log(1/\beta)\cdot D\Big) \leq \frac{1}{2}$$ 
        such that  
    \begin{equation}\label{eq:main-thm-convergence}
        \mathbb{E}\left(\left\|\bm{X}_T - \target\right\|^2 1_{\Omega^c}\right) \leq \left(1-\frac{c_2}{n}\right)^T \left\|\init - \target\right\|^2.
    \end{equation}
   \item (Convergence in probability) We have 
   \begin{align} \label{eq: main failure-prob-Gaussian}
       \mathbb{P}\left(\|\bX_T-\bx^*\|^2\le \Big(1-\frac{c_3}{n}\Big)^T\|\bx_0-\bx^*\|^2\right) \ge 1 - T^{-5} - 2\exp\Big(-\frac{c_4T}{n}\Big).
   \end{align}
   \end{itemize}
    \end{theorem}

We note that \cref{eq: main failure-prob-Gaussian} is what we claimed as the first informal theorem in  \cref{sec11}. 
 Due to Markov's inequality, \cref{eq: main failure-prob-Gaussian} is indeed an immediate consequence of the expected convergence \cref{eq:main-thm-convergence}. 

\begin{proposition}
    In  \cref{thm: main Gaussian}, the convergence in expectation \cref{eq:main-thm-convergence} implies the high-probability one \cref{eq: main failure-prob-Gaussian}. \label{pro1} 
\end{proposition}
\begin{proof}
    
Since $\mathbb{P}(\Omega^c) \geq \frac{1}{2}$, we have
    \begin{align*}
        \mathbb{E}\big(\|\bm{X}_T-\target\|^2 1_{\Omega^c}\big) = \mathbb{P}(\Omega^c)\cdot\mathbb{E}\big(\|\bm{X}_T-\target\|^2\big|\Omega^c\big) \ge \frac{1}{2}\mathbb{E}\big(\|\bm{X}_T-\target\|^2\big|\Omega^c\big),
    \end{align*}
    which gives
    \begin{align*}
        \mathbb{E}\big(\|\bm{X}_T- \target \|^2\big|\Omega^c\big)\le 2 \Big(1-\frac{c_2}{n}\Big)^T \|\bm{x}_0-\target\|^2.
    \end{align*}
    Given $\epsilon  > 0$, by Markov's inequality we have
    \begin{align*}
        \mathbb{P}\Big(\|\bm{X}_T-\target\|^2 \ge \epsilon \|\bm{x}_0-\target\|^2\Big|\Omega^c\Big) \le \frac{2(1-\frac{c_2}{n})^T\|\bm{x}_0-\target\|^2}{\epsilon\|\bm{x}_0-\target\|^2} = \frac{2(1-\frac{c_2}{n})^T}{\epsilon},
    \end{align*}
    and hence
    \begin{align*}
         & \mathbb{P}\Big(\|\bm{X}_T -\target\|^2 \ge \epsilon\|\bm{x}_0-\target\|^2\Big) \\&\le \mathbb{P}(\Omega^c)   \mathbb{P}\Big(\|\bm{X}_T -\target\|^2 \ge \epsilon\|\bm{x}_0-\target\|^2\Big|\Omega^c\Big)+\mathbb{P}(\Omega)\\
         & \le \frac{2(1-\frac{c_2}{n})^T}{\epsilon} + \mathbb{P}(\Omega).
    \end{align*}
    To obtain the desired result \cref{eq: main failure-prob-Gaussian}, we set $\epsilon = (1 - \frac{c_2}{2n})^T$ and note that  
    \[
    \frac{2(1 - c_2/n)^T}{\epsilon} \leq 2 \left( \frac{1 - c_2/n}{1 - c_2/(2n)} \right)^T \leq 2 \Big(1 - \frac{c_2}{2n}\Big)^T \leq 2 \exp\Big(-\frac{c_2T}{2n}\Big),
    \]
    where the last step uses $\ln(1-x) \leq -x$. Setting $c_3 = c_2/2$ and $c_4 = c_2/2$ yields the result in \cref{eq: main failure-prob-Gaussian}.
\end{proof}

Therefore, all that remains is to prove the convergence in expectation in \cref{eq:main-thm-convergence}. To that end, we require some new ideas beyond \cite{haddock_quantilebased_2022,haddock_subsampled_2023,steinerberger2023quantile}.
We shall especially compare to the proof of \cite[Thm. 1.1]{haddock_quantilebased_2022} for an algorithm nearly identical to \cref{algo:quantileRK}, with the only difference being the replacement of our subsampled quantile by the full-sample one $Q_q(\bx_k)$.\footnote{In contrast,  \cite{steinerberger2023quantile,haddock_subsampled_2023} analyzed algorithms that select the update sample from acceptable rows to avoid the potential rejection of some update.} Before presenting the proof of \cref{thm: main Gaussian} in the next subsection, we shall pause to summarize these new ideas as our technical contributions.

\subsubsection{New Ideas for Proving (\ref{eq:main-thm-convergence})}\label{sec:newideas}
Prior works relied on \cref{lem: full quantile} to relate the full-sample quantile to the uncorrupted quantile $\tilde{Q}$ (which can be further bounded by \cref{lem: A and ideal quantile}). Such a relation no longer holds arbitrarily for quantiles computed from $D$ samples with $D\ll \beta m$, as the entire quantile subsample can be corrupted and the subsampled quantile can take an arbitrary value. Fortunately, we can still bound our subsampled quantile from both sides by the uncorrupted quantile, with high probability, using the Chernoff bound.

\begin{lemma} \label{lem: sub quantile}
Given $\bm{x}_k$, $q \in (0,1)$, $\e{l} \in (0,q)$, $\e{u} \in (0,1-q)$ and $\beta \in (0, \min\{q - \e{l}, 1-q - \e{u}\})$, 
    with respect to the randomness in the selection of subsample $\subsample{i}{k+1}$, we have the following statements: 
    \begin{itemize}
    \item {\it (Upper bound)}  With probability at least $1- \exp(- \D(q \| q + \e{u})  D)$, we have
\begin{equation} \label{eq: sub quantile upper bound}
        Q_q\left(\bm{x}_k,\subsample{i}{k+1}\right) \leq \tilde{Q}_{q+\beta + \e{u}}\left(\bm{x}_k\right).
    \end{equation} 

    \item  {\it (Lower bound)} With probability at least $1- \exp(- \D(q \|q - \e{l})  D)$, we have
    \begin{equation} \label{eq: sub quantile lower bound}
       Q_q\left(\bm{x}_k,\subsample{i}{k+1}\right)\ge\tilde{Q}_{\frac{q-\beta - \e{l}}{1 - \beta}}\left(\bm{x}_k, B^c\right) \ge \tilde{Q}_{q-\beta - \e{l}}\left(\bm{x}_k\right).  
    \end{equation}
        \end{itemize} 
\end{lemma}

\begin{proof}

   For the lower bound of $Q_q\left(\bm{x}_k,\subsample{i}{k+1}\right)$, it suffices that more than $(1-q)D$ indices in $\subsample{i}{k+1}$ fall in the set
    \[
        \underline{S}_{l} := \left\{ i \in B^c : \left|\left\langle \bm{x}_k - \bm{x}^*, \bm{a}_i \right\rangle\right| \geq \tilde{Q}_{\frac{q-\beta - \e{l}}{1 - \beta}}\left(\bm{x}_k, B^c\right) \right\},
    \]
    which contains at least $(1 - q + \e{l}+\beta)\frac{|B^c|}{1-\beta} - \beta m = (1 - q + \e{l})m$ indices. Let $Y_j$ indicate whether $i_j^{(k+1)}$ belongs to these $(1 - q + \e{l})m$ indices,  then $\sum_{j=1}^D Y_j \sim \Bin{D}{1 - q + \e{l}}$. By \cref{thm: chernoff}, we have
    \[
    \begin{aligned}
        \mathbb{P}\left( \sum_{j=1}^D Y_j \leq (1 - q) D \right)
        &\leq \exp(- \D(1 - q \| 1 - q + \e{l}) D), \\
        \D(1 - q \| 1 - q + \e{l}) &= \D( q \| q - \e{l}) > 0.
    \end{aligned}
    \]
    Moreover, observe that $\tilde{Q}_{q-\beta - \e{l}}\left(\bm{x}_k\right) \leq \tilde{Q}_{\frac{q-\beta - \e{l}}{1 - \beta}}\left(\bm{x}_k, B^c\right)$, since the $(q - \beta - \e{l})m$th smallest ideal residual among all indices $[m]$ cannot exceed the $(q - \beta - \e{l})m$th smallest among the uncorrupted set $B^c$.

    For the upper bound, it suffices that at least $q D$ indices in $\subsample{i}{k+1}$ fall in
     \[
        \overline{S}_{u} := \left\{i \in B^c: \left|\left\langle \bm{x}_k - \bm{x}^*, \bm{a}_i \right\rangle\right| \leq \tilde{Q}_{q+\beta+\e{u}}\left(\bm{x}_k\right) \right\},
        \]
    which has cardinality at least $(q+\beta+\e{u})m-\beta m=(q+\e{u})m$. Let $Y_j$ indicate whether $i_j^{k+1}$ belongs to these $(q + \e{u})m$ indices, then $\sum_{j=1}^D Y_j \sim \Bin{D}{q + \e{u}}$. By \cref{thm: chernoff}, we have
    \[
        \mathbb{P}\left( \sum_{j=1}^D Y_j  \leq  q D \right)
        \leq  \exp(- \D(q \| q + \e{u}) D).
    \]
    The proof is complete. %
\end{proof}

Our overall approach for proving \cref{eq:main-thm-convergence} remains similar to prior works, in that the two central steps are controlling the impact of selecting a corrupted row and utilizing the expected contraction induced by accepting an uncorrupted row (due to the Strohmer--Vershynin bound). Nonetheless, each of these two steps involves new technicalities, as we present below.

The key to controlling the impact of the update sample being  corrupted  is an appropriate upper bound on the subsampled quantile. As mentioned, this does not hold arbitrarily, and we shall rely on the probabilistic upper bound \cref{eq: sub quantile upper bound} from \cref{lem: sub quantile}.  Focusing on the first $T$ iterations, we advocate to work with a failure event $\Omega$ defined as \cref{eq: sub quantile upper bound} fails for some $k < T$, whose probability is bounded by a union bound; see \cref{lem: failure probability small}.

Then we seek to bound $ \mathbb{E}(\left\|\bm{X}_T - \bx^*\right\|^2 1_{\Omega^c})$. Due to the indicator function, it is unclear what per-iteration contraction we can use to obtain the desired bound. Specifically, 
previous works \cite{haddock_quantilebased_2022,haddock_subsampled_2023,steinerberger2023quantile} showed $\mathbb{E}(\|\bX_{k+1}-\bx^*\|^2)\le (1-\frac{c}{n})\|\bx_{k}-\bx^*\|^2$ and iterated this bound. Unfortunately, this is no longer feasible in our subsampled setting.\footnote{To see this, one may use   \cref{lem: corrupted error}--\cref{lem: prob-disaster-event} to reach the lower bound
\begin{align*}
    \mathbb{E}\big(\|\bX_{T+1}-\bx^*\|^2\big) \ge \Big(\frac{\beta}{2}\Big)^{D+1}\min_{i\in B}|\epsilon_i|^2
\end{align*}
and then set the corruption large enough. 
}
Our remedy is to define an event $S_k$ for the $k$th iteration in \cref{eq: def for Sk_relax} and a stopping time $\tau$ in \cref{eq:def_tau2 3}, and then relate these to $\Omega^c$ in \cref{eq:tau2 and Sk}. Further using some concepts from stochastic process theory \cite{durrett2019probability}, we observe that the one-step contraction bound on $\mathbb{E}\big[\|\bX_{k+1}-\bx^*\|^21_{S_{k+1}}\big]$ in \cref{ifwehave} can be iterated to achieve our goal; see \cref{lem: induction}. 
Similar arguments were employed in \cite{tan_phase_2019} to address a similar issue in phase retrieval---the error contraction was only established within a neighborhood of the true signal, yet the randomized Kaczmarz iterates do not reside in this neighborhood arbitrarily. 

Therefore, it remains to establish \cref{ifwehave} in which the contraction stems from an application of the Strohmer--Vershynin bound to the case of accepting an uncorrupted row. Note that exactly $qm$ rows are acceptable under the full-sample quantile, the authors of \cite{haddock_quantilebased_2022}  could easily find a set of acceptable 
uncorrupted rows of size $(q-2\beta)m$ (called $I_2$ therein) and then apply the Strohmer--Vershynin bound to the uncorrupted linear system constituted by the rows in $I_2$. However, this becomes more intricate under our subsampled quantile, since it is much less clear which uncorrupted rows will be accepted. Our idea is to identify {\it a large enough uncorrupted subset of the acceptable rows}. To find such a subset, we first establish a probabilistic lower bound on the subsampled quantile by \cref{eq: sub quantile lower bound} from \cref{lem: sub quantile} and then consider the uncorrupted rows with residuals smaller than such a lower bound; see the lower bound event in \cref{lowerboundevent} and the set $S$ in \cref{eq:def_S_set}. %
We manage to ensure $|S|=\Omega(m)$ by carefully setting the relevant parameters.  

\subsection{Proof of Theorem~\ref{thm: main Gaussian}} \label{sec: Proof of the main} 
In this subsection, we first introduce the necessary definitions and lemmas, and then prove \cref{thm: main Gaussian}.

We define the failure event $\Omega$ as 
\begin{align} 
    \Omega:= \big\{\text{\cref{eq: sub quantile upper bound} fails for some $0 \leq k \leq T-1$}\big\}.\label{defineOmega}
\end{align} 
Applying a union bound over $T$ iterations, we see that the probability of $\Omega$ can be made arbitrarily small by increasing $D$.
\begin{lemma} \label{lem: failure probability small}
We have 
$$
\mathbb{P}(\Omega) \leq T \exp( -\D(q \| q + \e{u}) D) ,
$$
 where $\e{u}$ is a constant as in \cref{lem: sub quantile}. 
\end{lemma}
\begin{proof}
    $$
    \begin{aligned} \label{eq: failure probability}
        \mathbb{P}(\Omega) & =  \mathbb{P}( \cup_{k=0}^{T-1} \{\mbox{\cref{eq: sub quantile upper bound} fails at step $k$} \} )                                             \\
                                  & \leq \sum_{k=0}^{T-1} \mathbb{P}(\{\mbox{\cref{eq: sub quantile upper bound} fails at step $k$} \} )                  \\
                                  &\stackrel{\note{a}}{\leq}   T \exp( -\D(q \| q + \e{u}) D) ,                                               \\
    \end{aligned}
$$ 
where (a) follows by \cref{lem: sub quantile}.
\end{proof}
We define the event that \cref{eq: sub quantile upper bound} holds for step $k$ as
\begin{equation} \label{eq: def for Sk_relax}
    S_{k+1} := \left\{ Q_q(\bm{X}_{k}, \subsample{i}{k+1}) \leq \tilde{Q}_{q + \beta + \e{u}}(\bm{X}_{k})\right\},
\end{equation}
as well as introduce the stopping time as
\begin{equation} \label{eq:def_tau2 3}
    \tau := \inf_{ t \geq 1} \left\{ t : Q_q(\bm{X}_{t-1}, \subsample{i}{t}) > \tilde{Q}_{q + \beta + \e{u}}(\bm{X}_{t-1}) \right\}.
\end{equation}
These conventions can further characterize $\Omega$ in \cref{defineOmega} as  
\begin{align} \label{eq:tau2 and Sk}
\{\tau > k + 1\} = \{\tau > k, S_{k+1} \mbox{ holds}\} \subset \{\tau > k\}\qquad\text{and}  \qquad \Omega^c = \{\tau > T\}.
\end{align}
These observations allow us to bound $ \mathbb{E}(\left\|\bm{x}_T - \bx^*\right\|^2 1_{\Omega^c})$ by iterating an appropriate per-iteration contraction bound.
To formalize this, we introduce the standard filtration from stochastic processes~\cite{durrett2019probability}. For each $k$, let $\mathcal{F}_k$ denote the $\sigma$-algebra capturing the history up to $\bX_{k}$, generated by all update and quantile sample selections:
\begin{align*}
    \updateid{1},\, i^{(1)}_1, \ldots, i^{(1)}_D;\;
    \updateid{2},\, i^{(2)}_1, \ldots, i^{(2)}_D;\;
    \ldots;\;
    \updateid{k},\, i^{(k)}_1, \ldots, i^{(k)}_D, \label{Fk}
\end{align*}
where these random variables are sampled independently and uniformly from \\ $\{1, \ldots, m\}$. We then have:
\begin{lemma} \label{lem: induction}
If there exists a constant $c_2 > 0$ such that 
\begin{equation}
    \label{ifwehave}
    \mathbb{E}\left[
        \|\bm{X}_{k+1}-\target\|^2 1_{S_{k+1}}
        \mid \bm{X}_k = \bm{x}_k
    \right] \leq \left(1 - \frac{c_2}{n}\right) \|\bm{x}_k - \target\|^2
\end{equation}
holds for all fixed $\bx_k$ and $0\leq k \leq T-1$,
then we have
$$
  \mathbb{E}\left(\left\|\bm{X}_T - \target\right\|^2 1_{\Omega^c}\right) \leq \left(1-\frac{c_2}{n}\right)^T \left\|\init - \target\right\|^2.
  $$
  
\end{lemma}
\begin{proof}
For any $0\le k\le T-1$, we have 
\begin{equation} \label{eq: induction}
    \begin{aligned}
        & \mathbb{E}\left[\|\bm{X}_{k+1}-\target\|^2 1_{\tau>k+1} \right] \\
        & \stackrel{\note{a}}{=} \mathbb{E}\left[\|\bm{X}_{k+1}-\target\|^2 1_{\tau>k} 1_{S_{k+1}}\right] \\
        &\stackrel{\note{b}}{\leq} \mathbb{E}\left[
            \mathbb{E}\left[
                \|\bm{X}_{k+1}-\target\|^2 1_{S_{k+1}}
                \mid  \mathcal{F}_k
            \right] 1_{\tau>k}
        \right] \\
        &\stackrel{\note{c}}{\leq} \mathbb{E}\left[
            \mathbb{E}\left[
                \|\bm{X}_{k+1}-\target\|^2 1_{S_{k+1}}
                \mid \bm{X}_k
            \right] 1_{\tau>k}
        \right] \\
         &\stackrel{\note{d}}{\leq} \left(1 - \frac{c_2}{n}\right) \mathbb{E}\left[ \|\bX_{k}-\target\|^2
            1_{\tau>k}
        \right],
    \end{aligned}
\end{equation}
where (a) follows from \cref{eq:tau2 and Sk}, (b) follows from the tower property of conditional expectation and the fact that $1_{\tau > k}$ is measurable with respect to $\mathcal{F}_k$, (c) follows from the Markov property of $\{\bm{X}_k\}$ (i.e., $\bX_{k+1}$ depends on $\calF_k$ only through $\bX_k$), (d) follows from \cref{ifwehave}. 
Now further using $\Omega^c=\{\tau>T\}$, we can iterate \cref{eq: induction} to arrive at
\begin{align*}
  &\mathbb{E}\left(\left\|\bm{X}_T - \target\right\|^2 1_{\Omega^c}\right) = 
  \mathbb{E}\left(\left\|\bm{X}_T - \target\right\|^2 1_{\tau>T}\right) \le \Big(1-\frac{c_2}{n}\Big)  \mathbb{E}\left(\left\|\bm{X}_{T-1} - \target\right\|^2 1_{\tau>T-1}\right) \\
  &\le  \Big(1-\frac{c_2}{n}\Big)  ^2\mathbb{E}\left(\left\|\bm{X}_{T-2} - \target\right\|^2 1_{\tau>T-2}\right) \le \cdots \le \Big(1-\frac{c_2}{n}\Big)^T\|\bx_0-\bx^*\|^2,
\end{align*}
as desired. 
\end{proof}

    
With the above preparations, we are in a position to prove \cref{thm: main Gaussian}.
\begin{proof}[Proof of \cref{thm: main Gaussian}]
Before proceeding to the main arguments, we pause to define the constants involved in our proof. 

\textbf{Defining the Constants:} 
By \cref{assum1}, 
 \cref{eq: Gaussian} holds for $\alpha_0 = q/2$, and this leads to
\begin{equation} \label{eq:constants}
    \frac{\sigma_{\max}(\bm{A}) \sqrt{n}}{\sqrt{m}} \leq C \quad \text{and} \quad \smin{\alpha_0} \geq \frac{\sqrt{2\pi}(q/2)^{3/2}}{24} > 0,
\end{equation}
  where $C$ is an absolute constant.
We then choose $\e{l} = q/4 \in (0, q - \alpha_0)$  and define
\begin{equation}
    p_l := \exp(-\D(q \|  q - \e{l})  D), \qquad p_l^c := 1 - p_l ,
    \label{eq:pl-plc}
\end{equation}
where $\D(q \|  q - \e{l}) = \D(q \|  \frac{3}{4}q)$ is a positive constant depending on $q$ only. Also, by $D\ge 1$ we have $p_l\le \exp(-\D(q\|\frac{3}{4}q))$ and 
\begin{equation}  \label{eq:pl-plcDge1}
    p_l^c\ge 1- \exp(-\D(q\|\frac{3}{4}q)).
\end{equation}
 We assume
\begin{equation}
    \beta < q - \alpha_0 - \e{l} = q/4,
    \label{eq: valid conditions beta low}
\end{equation}
and note that this, along with the definition in \cref{eq:uniform-smin}, leads to
\begin{equation} 
    \alpha := q - \beta - \e{l} > \alpha_0, \qquad \smin{\alpha} \geq \smin{\alpha_0} > 0.
    \label{eq:alpha-lowerbound}
\end{equation}
We also define 
\begin{equation} \label{eq: C_1}
    C_1 := \max\left\{ \frac{(\sigma_{\max}(\bm{A}) \sqrt{n/m})^2}{\frac{1}{4}p_l^c \sigma_{\min, \alpha_0}^2}, \left(\frac{2(\sigma_{\max}(\bm{A}) \sqrt{n/m})^2}{\frac{1}{4}p_l^c \sigma_{\min, \alpha_0}^2}\right)^2 \right\} \leq C_1', \alpha' := C_1' \beta,
\end{equation}
where $C_1'$ can be chosen as a constant depending only on $q$ from \cref{eq:constants,eq:pl-plcDge1}.
Assume $\beta$ is small enough such that
\begin{equation}
    \beta \le \frac{1-q}{2(C_1' + 1)},
    \label{eq: valid conditions beta upper}
\end{equation}
then we have
\begin{equation} \label{eq:def eu}
    \e{u} := 1 - q - \beta - \alpha' > 0, \quad \beta + \alpha' = (1 + C_1')\beta = \frac{1 - q}{2} - \e{u}< 1-q.
\end{equation}

\textbf{Bounding $\mathbb{P}(\Omega)$:} 
Recall that the failure event $\Omega$, $S_k$, and $\tau$ have been defined based on the constants $q,\beta,\e{u}$; see \cref{defineOmega}, \cref{eq: def for Sk_relax} and \cref{eq:def_tau2 3}. 
$\mathbb{P}(\Omega)$ can be upper bounded by \cref{lem: failure probability small} and by definition of $\D(q \| q + \e{u})$:
\begin{align*}
    \mathbb{P}(\Omega) & \leq T \exp\left(-\D(q \| q + \e{u}) D\right) \\
    &\stackrel{\note{a}}{=} T  \exp\left(-\left( (1-q) \ln \frac{1-q}{(1+C_1')\beta} + q \ln \frac{q}{1-(1+C_1')\beta} \right) D\right) \\
    &\stackrel{\note{b}}{=} T  \exp\left(-\left( (1-q) \ln \frac{1}{\beta} - C_2' \right) D\right) \\
    &\stackrel{\note{c}}{\leq} T  \exp\left(-\left( \frac{1-q}{2} \ln \frac{1}{\beta} \right) D\right), \label{eq: constant q and D}
\end{align*}
where (a) is due to \cref{eq:def eu}, in (b) we define $C_2' := - (1-q) \ln (\frac{1-q}{1+C_1'}) - q \ln(\frac{q}{1-(1+C_1')\beta})$ that depends on $q$ and $\beta$, 
and in (c) we utilize the factor $\log(\frac{1}{\beta})$ and  notice that for some small enough $c_q>0$ depending on $q$ only, for all
\begin{equation}\label{eq:beta-upperbound-lnabsorb}
    \beta < c_q,
\end{equation}
we have
$
(1-q) \ln \frac{1}{\beta} - C_2' > \frac{1}{2}(1-q)\ln\frac{1}{\beta} > 0.
$
Moreover, there exists a constant $C_2 = \frac{24}{1-q}$, such that if $D \geq C_2 \frac{\log T}{\log(1/\beta)}$ and $T$ is an integer greater than $1$, then
\begin{equation}
    \mathbb{P}(\Omega) \leq T^{-5} \exp\left(-\frac{1-q}{4} \log(1/\beta) \cdot D\right) \leq \frac{1}{2}.
    \label{eq:main-failure-prob in proof}
\end{equation}
 
            
\textbf{Bounding $\mathbb{E}(\left\|\bm{X}_T - \bx^*\right\|^2 1_{\Omega^c})$}: By \cref{lem: induction}, it suffices to show a per-iteration contraction:
\[
    \mathbb{E}\left[
        \|\bm{X}_{k+1}-\target\|^2 1_{S_{k+1}}
        \mid \bm{X}_k = \bm{x}_k
    \right] \leq \left(1 - \frac{c_2}{n}\right) \|\bm{x}_k - \target\|^2
\]
for some $c_2 > 0$. We proceed by conditioning on whether or not the update sample $\updateid{k+1}$ is corrupted:
 \begin{align}
        & \mathbb{E}\left[\|\bm{X}_{k+1}-\target\|^2 1_{S_{k+1}} \mid \bm{X}_k = \bm{x}_k\right] \nonumber\\
        &= \beta\, \mathbb{E}\left[\|\bm{X}_{k+1}-\target\|^2 1_{S_{k+1}} \mid \bm{X}_k = \bm{x}_k,\, \updateid{k+1} \in B\right] \nonumber\\
        &\quad + (1 - \beta)\, \mathbb{E}\left[\|\bm{X}_{k+1}-\target\|^2 1_{S_{k+1}} \mid \bm{X}_k = \bm{x}_k,\, \updateid{k+1} \in B^c\right] \nonumber \\
        &:= \beta I_1 + (1 - \beta) I_2. \label{eq: expectation_decompose_beta}
    \end{align}

    \subsubsection*{Case I: Corrupted Update Sample $\updateid{k+1} \in B$}

  With the aid of $1_{S_{k+1}}$, we can derive bound for  
    $I_1  := \mathbb{E}\left[\left\|\boldsymbol{X}_{k+1}-\target\right\|^2 1_{S_{k+1}} \mid \boldsymbol{X}_k=\boldsymbol{x}_k, \updateid{k+1}\in B\right].$ We first expand the expectation by conditioning on the specific update row $i \in B$ and the quantile subsample $\subsample{(i')}{k+1}$, and recall that the summation $\sum_{\subsample{(i')}{k+1}}$ is taken over all $m^D$ possible quantile subsamples:
    \begin{align*}
            I_1 
                & = \frac{1}{\beta m}\sum_{i \in B} \frac{1}{m^{D}}\sum_{\subsample{(i')}{k+1}} \\
                &\mathbb{E}\left[ \|\be_{k+1}\|^2 1_{S_{k+1}}\mid \bm{X}_k=\bm{x}_k, \updateid{k+1} = i, \subsample{i}{k+1} =\subsample{(i')}{k+1}\right].
    \end{align*}
    By the definition of $S_{k+1}$, the indicator $1_{S_{k+1}}$ ensures the subsampled quantile is upper bounded by $\tilde{Q}_{q+\beta+\e{u}}(\bm{x}_k)$. We then apply \cref{lem: quantile bound error} to obtain
    \begin{align*}
         I_1 & \leq \sum_{i \in B} \sum_{\subsample{(i')}{k+1}}
            \left( \left\|\bm{e}_k\right\|^2 + \tilde{Q}_{q+\beta+\e{u}}(\bm{x}_{k})^2 + 2 \tilde{Q}_{q+\beta+\e{u}}(\bm{x}_{k}) \left|\left\langle \bm{a}_{i},\bm{e}_k\right\rangle\right| \right) \frac{1}{\beta m^{D+1}}.
    \end{align*}
    We then apply \cref{lem: A and ideal quantile} with $q' = q + \beta + \e{u} = 1 - \alpha'$ to bound the ideal quantile term. This gives $\tilde{Q}_{q+\beta+\e{u}}(\bm{x}_{k}) \leq \frac{\Phi_{1 - \alpha'}}{\sqrt{n}} \|\bm{e}_k\|$, where $\Phi_{1 - \alpha'} = \sigma_{\max}(\bm{A}) \sqrt{n/m} / \sqrt{\alpha'}$. Substituting this bound and noting that the sum over subsamples cancels the $1/m^D$ factor, we obtain
    \begin{align*}
         I_1 & \leq \sum_{i \in B}  \left(\|\bm{e}_k\|^2 + \frac{(\Phi_{1 - \alpha'})^2}{n} \|\bm{e}_k\|^2
            + \frac{2(\Phi_{1 - \alpha'})}{\sqrt{n}} \|\bm{e}_k\|  |\langle\bm{a}_i, \bm{e}_k \rangle| \right) \frac{1}{\beta m}.
    \end{align*}
    Finally, we bound the term involving $|\langle\bm{a}_i, \bm{e}_k \rangle|$. By \cref{lem: absolute sum bound}, we have \\
    $
    \frac{1}{\beta m} \sum_{i \in B}  |\langle\bm{e}_k, \bm{a}_i\rangle| \leq \frac{\sigma_{\max}(\bm{A}) \sqrt{n/m}}{\sqrt{n \beta }}\|\bm{e}_k\|,
    $
    and we define $c_B := \sigma_{\max}(\bm{A}) \sqrt{n/m} / \sqrt{\beta}$. Combining all terms and factoring out $\|\bm{e}_k\|^2$, we arrive at
    \begin{equation} \label{eq:I1-last}
        I_1 \leq \left( 1 + \frac{\Phi_{1 - \alpha'}^2 + 2\Phi_{1 - \alpha'}c_B }{n } \right) \|\bm{e}_k\|^2.
    \end{equation}

    \subsubsection*{Case II: Uncorrupted Update Sample $\updateid{k+1} \in B^c$}

    In this case, the error is nonexpansive. To ensure contraction, we need to identify {\it a large enough uncorrupted subset of the acceptable rows} for applying \cref{Stbound}. According to \cref{algo:quantileRK}, acceptable rows have residuals smaller than the subsampled quantile.
    To facilitate analysis, we define $S_{k+1}^*$ as  the event that the subsampled quantile exceeds $\tilde{Q}_{\frac{q-\beta - \e{l}}{1 - \beta}}(\bm{x}_k, B^c)$: 
    \begin{align}\label{lowerboundevent}
          S_{k+1}^* := \left\{ Q_q\left(\bm{x}_{k},\subsample{i}{k+1}\right) \ge \tilde{Q}_{\frac{q-\beta - \e{l}}{1 - \beta}}\left(\bm{x}_{k}, B^c\right)  \right\}. 
    \end{align}
    By the lower bound in \cref{lem: sub quantile}, this event occurs with probability at least $1 - p_l$ ($p_l$ defined in \cref{eq:pl-plc}):
     \begin{equation} \label{eq: prob for Sk_lower}
       \quad \mathbb{P}\left(\left(S_{k+1}^*\right)^c\right) \leq p_l.
    \end{equation}
 In this part of analysis, $S^*_{k+1}$ is indeed a good event: having found the lower bound on the subsampled quantile, we can ensure rows with residuals smaller than this lower bound are acceptable. Also recall that for uncorrupted rows $i\in B^c$, the absolute residual can be written as $|\langle\bx_k-\bx^*,\ba_i\rangle|$.  
 Then,
 we define $S$ as the set of indices in $B^c$ with residuals bounded by $\tilde{Q}_{\frac{q-\beta - \e{l}}{1 - \beta}}\left(\bm{x}_k, B^c\right)$:
    \begin{equation}
        S := \left\{ i \in B^c : \left| \langle \bm{x}_k - \bm{x}^*, \bm{a}_i \rangle \right| \leq \tilde{Q}_{\frac{q-\beta - \e{l}}{1 - \beta}}\left(\bm{x}_k, B^c\right) \right\}
        \label{eq:def_S_set},
    \end{equation}
     and notice that $$|S| = \frac{(q - \beta - \e{l})}{1-\beta}\cdot (1-\beta)m = \alpha m \geq \alpha_0 m = \frac{qm}{2},$$ 
     which we assume to be an integer for simplicity. 
     Thus, $I_2$ defined in \cref{eq: expectation_decompose_beta} can be decomposed as follows:
    \begin{align}
        I_2 & =  \mathbb{E}\left[\left\|\boldsymbol{X}_{k+1}-\target\right\|^2 1_{S_{k+1}}  \mid \boldsymbol{X}_k=\boldsymbol{x}_k, \updateid{k+1}\in B^c\right]                              \nonumber  \\
        & \leq \mathbb{E}\left[\left\|\boldsymbol{X}_{k+1}-\target\right\|^2  \mid \boldsymbol{X}_k=\boldsymbol{x}_k, \updateid{k+1}\in B^c\right]                                \nonumber   \\
        & =
        \mathbb{E}\left[\left\|\boldsymbol{X}_{k+1}-\target\right\|^2  
        \mid \boldsymbol{X}_k=\boldsymbol{x}_k, \updateid{k+1} \in S \right] \frac{\alpha}{1 - \beta}  \nonumber\\
        & \quad + \mathbb{E}\left[\left\|\boldsymbol{X}_{k+1}-\target\right\|^2 \mid \boldsymbol{X}_k=\boldsymbol{x}_k, \updateid{k+1} \in B^c \setminus S \right] \left(1 - \frac{\alpha}{1 - \beta}\right),                                          \nonumber   \\
        &\stackrel{\note{a}}{\leq}
         \mathbb{E}\left[\left\|\boldsymbol{X}_{k+1}-\target\right\|^2  
        \mid \boldsymbol{X}_k=\boldsymbol{x}_k, \updateid{k+1} \in S \right] + \left\|\be_{k}\right\|^2 \left(1 - \frac{\alpha}{1 - \beta}\right),
        \label{eq:I2-medium}
    \end{align}
    where in (a) we use that the error is nonexpansive for $\updateid{k+1} \in (B^c \setminus S)\subset B^c$.
             Then we have:
   
    \begin{align*}
    &\mathbb{E}\left[\left\|\boldsymbol{X}_{k+1}-\target\right\|^2 
                \mid \boldsymbol{X}_k=\boldsymbol{x}_k, \updateid{k+1} \in S \right] \\
        & \stackrel{\note{a}}{=}
        \mathbb{E}\left[\left\|\boldsymbol{X}_{k+1}-\target\right\|^2 
                \mid \boldsymbol{X}_k=\boldsymbol{x}_k, \updateid{k+1} \in S,   S_{k+1}^*\text{ holds}\right] \mathbb{P}(S_{k+1}^*)\\
        & \quad \  + \mathbb{E}\left[\left\|\boldsymbol{X}_{k+1}-\target\right\|^2 
                \mid \boldsymbol{X}_k=\boldsymbol{x}_k, \updateid{k+1} \in S,  S_{k+1}^*\text{ fails}\right] \mathbb{P}((S_{k+1}^*)^c)  \\
         &\stackrel{\note{b}}{\leq}   \mathbb{E}\left[\left\|\bm{e}_{k} - \langle\bm{e}_k, \bm{a}_{\updateid{k+1}}\rangle \bm{a}_{\updateid{k+1}}\right\|^2 
                \mid \boldsymbol{X}_k=\boldsymbol{x}_k, \updateid{k+1} \in S\right] \mathbb{P}(S_{k+1}^*)\\
        & \quad\  + \mathbb{E}\left[\left\|\be_k\right\|^2 
                \mid \boldsymbol{X}_k=\boldsymbol{x}_k, \updateid{k+1} \in S\right] \mathbb{P}((S_{k+1}^*)^c) \\   
        &\stackrel{\note{c}}{\leq} 
        \mathbb{E} \left[ (1 - p_l) \left\|\bm{e}_{k} - \langle\bm{e}_k, \bm{a}_{\updateid{k+1}}\rangle \bm{a}_{\updateid{k+1}}\right\|^2  + p_l \|\be_k\|^2  \mid \bX_k = \bx_k, \updateid{k+1} \in S\right],
        \end{align*}
    where (a) is due to the independence of the event $S^*_{k+1}$ (which only hinges on the randomness of the quantile subsample) and the update sample $\updateid{k+1}$. In (b), when $\updateid{k+1} \in S$ and $S_{k+1}^*$ holds, the update is accepted and the error becomes $\left\|\be_{k+1}\right\|^2 = \left\|\be_k - \langle\be_k, \ba_{\updateid{k+1}}\rangle \ba_{\updateid{k+1}}\right\|^2$; otherwise, the error is nonexpansive since $\updateid{k+1} \in B^c$; after upper bounding $\|\boldsymbol{X}_{k+1}-\target\|^2$, we get rid of the conditioning on $S_{k+1}^*$ as it no longer has an impact. 
    In (c), we use \cref{eq: prob for Sk_lower} and $\left\|\be_k - \langle\be_k, \ba_{\updateid{k+1}}\rangle \ba_{\updateid{k+1}}\right\|^2 \leq \left\|\be_k\right\|^2$. Substituting the final estimation into \cref{eq:I2-medium} yields:
        \begin{align*}
                 I_2   & \leq
                \mathbb{E}\left[\left\|\bm{e}_{k} - \langle\bm{e}_k, \bm{a}_{\updateid{k+1}}\rangle \bm{a}_{\updateid{k+1}}\right\|^2 
                \mid \boldsymbol{X}_k=\boldsymbol{x}_k, \updateid{k+1} \in S \right] \frac{\alpha}{1 - \beta} p_l^c\\
                & \quad +   \mathbb{E}\left[\left\|\bm{e}_{k}\right\|^2 
                \mid \boldsymbol{X}_k=\boldsymbol{x}_k, \updateid{k+1} \in S \right] \frac{\alpha}{1 - \beta}p_l   \\
                    & \quad + \mathbb{E}\left[\left\|\bm{e}_{k} \right\|^2  \mid \boldsymbol{X}_k=\boldsymbol{x}_k, \updateid{k+1} \in S^c \right] \left(1 - \frac{\alpha}{1 - \beta}\right)                                             \\
                    & \leq \left(1 - \frac{\alpha p_l^c}{1 - \beta} \right)\left\|\bm{e}_k \right\|^2                                                                   +
                \underbrace{\mathbb{E}\left[\left\|\bm{e}_{k} - \langle\bm{e}_k, \bm{a}_{\updateid{k+1}}\rangle \bm{a}_{\updateid{k+1}}\right\|^2   \mid \boldsymbol{X}_k=\boldsymbol{x}_k,
                \updateid{k+1} \in S \right] }_{\leq
                 \left(1 - \frac{\sigma_{\min}^2\left(\bm{A}_{S}\right)}{\left\|\bm{A}_{S}\right\|_F^2}\right) \left\|\bm{e}_k \right\|^2 
                } \frac{\alpha p_l^c}{1 - \beta}, 
            \end{align*}
        where we apply the Strohmer--Vershynin bound in \cref{Stbound} to the submatrix $\bm{A}_{S}$ (with $|S| = \alpha m$) to establish the contraction in the last term. By definition \cref{eq:uniform-smin} and the lower bound in \cref{eq:alpha-lowerbound}, we have $\sigma_{\min}\left(\bm{A}_{S}\right) \geq \sigma_{\min, \alpha} \sqrt{\frac{m}{n}} \geq \sigma_{\min, \alpha_0} \sqrt{\frac{m}{n}}  > 0$.
        Moreover, $\|\bm{A}_{S}\|_F^2 = |S| = \alpha m$ since the rows of $\bA$ have unit norm.
         Therefore,
          \begin{align}
                 I_2   & \leq \left(1 - \frac{\alpha}{1 - \beta}p_l^c \right)  \left\|\bm{e}_k \right\|^2
                + \frac{\alpha}{1 - \beta}p_l^c \left(1 - \frac{\sigma_{\min, \alpha}^2}{\alpha}\frac{1}{n}\right)\left\|\bm{e}_k \right\|^2                                                                                                    \nonumber    \\
                & \leq \left(1 - p_l^c\frac{\alpha}{1 - \beta} \frac{\sigma_{\min, \alpha}^2}{\alpha} \frac{1}{n}\right) \left\|\bm{e}_k \right\|^2                                                                                       \nonumber      \\
                & = \left(1 - p_l^c \frac{\sigma_{\min, \alpha}^2}{1 - \beta} \frac{1}{n} \right)\left\|\bm{e}_k \right\|^2. \label{eq: I2 last}
            \end{align}   
    \textbf{Combining Cases I and II}: Substituting \cref{eq:I1-last,eq: I2 last} into \cref{eq: expectation_decompose_beta}, we have
     $$
        \begin{aligned}
             & \mathbb{E}\left[\left\|\boldsymbol{X}_{k+1}-\target\right\|^2 1_{S_{k+1}} \mid \boldsymbol{X}_k=\boldsymbol{x}_k\right]  \\
             & \leq \beta \left(\left\|\bm{e}_{k}\right\|^2+ \left(\Phi_{1 - \alpha'}^2 + 2\Phi_{1 - \alpha'} c_B \right)\frac{\left\|\bm{e}_{k}\right\|^2}{n} \right)  + \left(1 - \beta\right)
            \left(1 - p_l^c    \frac{\sigma_{\min, \alpha}^2}{1 - \beta} \frac{1}{n} \right)\left\|\bm{e}_k \right\|^2                                                                 \\
             & = \left(1 -  \frac{p_l^c\sigma_{\min, \alpha}^2 - \beta\left(\Phi_{1 - \alpha'}^2 + 2\Phi_{1 - \alpha'} c_B\right)}{n} \right) \left\|\bm{e}_k\right\|^2. \\
        \end{aligned}
        $$
    Combining the upper bounds for $\beta$ from \cref{eq:beta-upperbound-lnabsorb,eq: valid conditions beta low,eq: valid conditions beta upper}, we define $c_1 := \min\{q/4, \frac{1-q}{2(C_1' + 1)}, c_q\}$. Let $c_2 := \frac{1}{2} (1 -\exp(-\D(q \| \frac{3}{4} q )))\sigma_{\min, \alpha_0}^2 > 0$. %
    Both $c_1$ and $c_2$ are positive constants depending only on $q$. For $\beta \in (0, c_1)$, the following lower bound holds:
    \begin{equation*} \label{eq: valid conditions contration}
       p_l^c \sigma_{\min, \alpha}^2 - \beta(\Phi_{1 - \alpha'}^2 + 2\Phi_{1 - \alpha'} c_B) > p_l^c \sigma_{\min, \alpha}^2 - p_l^c \frac{1}{2} \sigma_{\min, \alpha_0}^2 \geq  c_2 > 0,
    \end{equation*}
    where the first inequality follows from the definition of $C_1'$ and $\alpha'$ in \cref{eq: C_1}, and the second inequality follows from \cref{eq:alpha-lowerbound,eq:pl-plcDge1}.
        This finishes the proof of \cref{eq:main-thm-convergence} by \cref{lem: induction}. Further combining with \cref{eq:main-failure-prob in proof}, we establish the high-probability convergence by \cref{pro1}.  
    \end{proof}

\subsection{Lower Bound}\label{sec: lower bound}
We assume the following on $\bA$, which is weaker than \cref{assum1} since we only require $\bA$ to satisfy the third condition in \cref{eq: Gaussian}. 
With high probability, it is satisfied by $\bA$ with rows being i.i.d. uniformly distributed over $\mathbb{S}^{n-1}$ when $m\ge C_1n$ for large enough $C_1$; see \cref{lem: randomness-Gaussian}. 
 \begin{assumption}\label{assump2}
      $\bA$ has rows in $\mathbb{S}^{n-1}$ and satisfies \begin{align}\label{assum2bound}
          \Phi_{q'}=\frac{\sigma_{\max}(\bA)\sqrt{n}}{\sqrt{m}\sqrt{1-q'}}\le \frac{C}{\sqrt{1-q'}},\qquad\forall q'\in(0,1)
      \end{align} 
      for some absolute constant $C.$ 
 \end{assumption}

 We now provide the formal theorem regarding the lower bound on $D$.  
\begin{theorem}\label{thm:lower}
 Consider problem \cref{eq: problem} with arbitrary $(\beta m)$-sparse $\bm{\epsilon}$ and suppose \cref{assump2} holds. Let $q \in (0,1)$, and the number of iterations $T$ be such that $T\ge  \lceil \frac{n}{\log n}\rceil$. For any $c_0\in(0,1)$, if the subsample size $D$ is a positive integer satisfying  
      $$D \le \max\Big\{\frac{c_0\log T}{\log(2/\beta)},1\Big\},$$
    then with probability at least $1- \frac{C_1}{\log n}-\exp(-\frac{\beta^2 n}{4\log n})-\exp(-\frac{\beta}{2}\frac{n}{T^{c_0}\log n})$ we have $$ \|\bX_T - \target\|^2 \geq ( 0.99 )^{\lfloor \frac{n}{\log n}\rfloor} \min_{i \in B} |\epsilon_i|^2.$$
\end{theorem}

As a consequence, under large enough $\min_{i\in B}|\epsilon_i|$, $\|\bX_T - \target\|^2$ can be made arbitrarily large with $1-o(1)$ probability as long as 
\begin{align} \label{scaling}
    \frac{\log n}{\beta^2n}=o(1)\qquad\text{and}\qquad \frac{T^{c_0}\log n}{\beta n}=o(1).
\end{align}
Note that such scaling assumptions are mild when $c_0\in(0,1)$ is small. For instance, if $\beta$ is a given positive constant, then  \cref{scaling}  holds as long as $T\le n^{\xi}$ for some $\xi<c_0^{-1}$, where $c_0^{-1}$ is a large constant. It also encompasses some settings with $\beta = o(1)$, with a specific example being
$$\beta= \Theta(n^{-\xi})\qquad\text{and}\qquad T\le n^{\xi_1}$$
 for some positive $\xi,\xi_1$ satisfying $\xi<\frac{1}{2}$ and $\xi_1<\frac{1-\xi}{c_0}$. In addition, when $\frac{\log T}{\log(2/\beta)}$ is small enough, our result ensures the failure of QRK under $D=1$ (recall that when $qD<1$ we will let the subsampled quantile be the smallest residual in the subsample).

Note that   \cref{thm: main Gaussian} and \cref{thm:lower} are both valid for arbitrary $(\beta m)$-sparse corruption, while there is an interesting nuance: \cref{thm: main Gaussian} is a stronger statement under corruption chosen by an adversary who has full knowledge of $(\bA,\bx^*)$,   stating that QRK with $D\ge C\frac{\log T}{\log(1/\beta)}$ can handle adversarial corruption; nonetheless, \cref{thm:lower} would be stronger when considering corruption oblivious to $(\bA,\bx^*)$---such as one that corrupts the first $\beta m$ measurements---as it states that QRK with $D\le c\frac{\log T}{\log(1/\beta)}$ fails even under such   less detrimental corruption.

\subsection{Proof of \cref{thm:lower}}\label{sec:lowerproof}

The approach to proving \cref{thm:lower} is to first identify a step at which the iterate is projected onto a corrupted row (we will use the last one over the first $T$ iterations) and then demonstrate that the subsequent iterations cannot significantly reduce the approximation error. The following lemma describes the minimum impact of projecting onto a corrupted row.

\begin{lemma} \label{lem: corrupted error}
    If $\bX_{k+1}$ is obtained by projecting $\bx_k$ onto a corrupted row, which is equivalent to  
\begin{equation} \label{eq: corrupted_update_condition}
\updateid{k+1} \in B \quad \text{and} \quad \left|\left\langle \ba_{\updateid{k+1}}, \bx_{k} \right\rangle - b_{\updateid{k+1}}\right| \leq Q\left(\bm{x}_k,\subsample{i}{k+1}\right),
\end{equation}
then:
\begin{align} \label{eq: corrupted error}
\| \bX_{k+1} - \target \| \geq \min_{i \in B} |\epsilon_i|.
\end{align}
\end{lemma}  
\begin{proof}
 Under the condition, $\bX_{k+1}$ is projected onto the corrupted row with \\ $\langle \ba_{\updateid{k+1}}, \bX_{k+1} \rangle = b_{\updateid{k+1}}$, and we have
$$\left\| \bX_{k+1} - \target \right\| \geq \left|\left\langle \ba_{\updateid{k+1}}, \bX_{k+1} - \target \right\rangle \right| = \left|\epsilon_{\updateid{k+1}}\right| \geq \min_{i \in B} \left|\epsilon_i\right|,$$ 
 as desired. 
\end{proof}
Further, such an event occurs with probability lower bounded as follows.
\begin{lemma} \label{lem: prob-disaster-event}
Given $\bx_k$, on the randomness of the quantile subsample and update sample (for updating $\bx_k$), the event in \cref{eq: corrupted error} occurs with probability at least
$(\frac{\beta}{2})^{D+1}$.
\end{lemma}
\begin{proof}
By \cref{lem: corrupted error}, it suffices to show that  \cref{eq: corrupted_update_condition} occurs with probability at least $(\frac{\beta}{2})^{D+1}$. Given $\bx_k$, we can always find $B_{1},B_{2}\subset B$ such that $|B_1|\ge \frac{\beta m}{2}$, $|B_2|\ge\frac{\beta m}{2}$, and the residuals in $B_2$ are uniformly larger than those in $B_1$:
$$\left|\left\langle \ba_{i_1},\bx_k\right\rangle - b_{i_1}\right|\le \left|\left\langle \ba_{i_2},\bx_k\right\rangle - b_{i_2}\right|,\qquad\forall i_1\in B_1,~i_2\in B_2.$$
 A sufficient condition for \cref{eq: corrupted_update_condition} is that the update index $\updateid{k+1}$ is selected from $B_{1}$, and all indices in the quantile subsample $\subsample{i}{k+1}$ are from $B_{2}$. Since $\updateid{k+1}$ and $\subsample{i}{k+1}$ are chosen independently from $\Uniform$, this occurs with probability at least $(\frac{\beta}{2})^{D+1}$.
\end{proof}

After  \cref{eq: corrupted error} happens, the error may decrease in subsequent iterations when  uncorrupted rows are accepted. Hence, we need to study the overall error reduction in these iterations. This amounts to analyzing the performance limit of standard RK, and we note that 
a lower bound   $\mathbb{E}\|\bX_{k+1} - \target\|^2 \ge (1 - \frac{c}{n}) \mathbb{E} \|\bX_{k} - \target\|^2$ was established in \cite[Thm. 3]{strohmer2009randomized}; see also \cite{steinerberger2021randomized}. However, this does not immediately transfer to a probabilistic statement. To that end, we introduce an additional parameter $\kappa$ and develop the following lemma. 

 \begin{lemma} \label{lem: fast-contraction}
 Given $\bx_k$ and assume that $\bX_{k+1}$ is  obtained from $\bx_k$ by projecting onto an uncorrupted row. Suppose that \cref{assump2} holds. For $\kappa \in (0,1)$, the error satisfies 
\begin{equation} \label{eq:fast-contraction-error}
    \|\bX_{k+1} - \target\|^2 \geq  \left(1 - \frac{C}{\kappa n}\right) \|\be_{k}\|^2,
\end{equation}
 with probability at least $1 - \kappa$, where $C > 0$ is an absolute constant.
 \end{lemma}
\begin{proof}
Since $\bX_{k+1}$ is obtained from projecting $\bx_k$ onto an uncorrupted row, we have
\[
\left\|\bX_{k+1} - \target\right\|^2 = \left\|\be_{k} - \langle \ba_{\updateid{k+1}}, \be_k \rangle \ba_{\updateid{k+1}} \right\|^2 = \left\|\be_k\right\|^2 - \left|\langle \ba_{\updateid{k+1}}, \be_k \rangle\right|^2, 
\]
i.e., that the reduction in error at each step depends on $\left|\langle \ba_{\updateid{k+1}}, \be_k \rangle\right|$. 
Let $\kappa > 0$ and assume $|\langle \ba_{\updateid{k+1}}, \be_k \rangle| \le  \tilde{Q}_{1 - \kappa}(\bm{x}_{k}) $, which holds with probability $1-\kappa$.
Therefore, we have  
\begin{equation}
\left\|\bX_{k+1} - \target\right\|^2 \geq \left\|\be_k\right\|^2 - \tilde{Q}_{1 - \kappa}\left(\bm{X}_{k}\right)^2  \stackrel{\note{a}}{\ge} \left(1-\frac{\Phi_{1-\kappa}^2}{n}\right)\left\|\be_k\right\|^2\stackrel{\note{b}}{\ge} \left(1 - \frac{C}{\kappa n}\right) \left\|\be_k\right\|^2,
\label{eq:lower-bound-error-decrease}
\end{equation}   
where (a) follows from \cref{lem: A and ideal quantile} and (b) follows from  \cref{assum2bound} in \cref{assump2}. 
\end{proof}



We are now ready to prove \cref{thm:lower}. 
 \begin{proof}[Proof of \cref{thm:lower}]
We let $k^*$ be the largest $k\le T-1$ such that $\bX_{k^*+1}$ is obtained from $\bX_{k^*}$ by projecting onto some corrupted row. We first show that $k^*\ge T - \lceil \frac{n}{\log n}\rceil$ holds with high probability. To see this, we notice that $k^*<T - \lceil \frac{n}{\log n}\rceil$ implies that the last $\lceil\frac{n}{\log n}\rceil$ iterations (among the first $T$ iterations) are not projections onto corrupted rows. Hence, by \cref{lem: prob-disaster-event} and the independence among the quantile subsample and update sample in different iterations, we have 
\begin{align}
    \mathbb{P}\Big(k^* <T-\Big\lceil\frac{n}{\log n}\Big\rceil\Big)\le \Big(1-\big(\frac{\beta}{2}\big)^{D+1}\Big)^{\frac{n}{\log n}}\le \exp\Big(-\big(\frac{\beta}{2}\big)^{D+1}\frac{n}{\log n}\Big).
    \nonumber
\end{align}
Further using $D\le \max\{\frac{c_0\log T}{\log(2/\beta)},1\}$, which gives $(\frac{\beta}{2})^D\ge \min\{T^{-c_0},\frac{\beta}{2}\}$, we reach 
\begin{align}
   \mathbb{P}\Big(k^* <T-\Big\lceil\frac{n}{\log n}\Big\rceil\Big)\le \exp\Big(-\frac{\beta^2n}{4\log n}\Big) + \exp\Big(-\frac{\beta}{2}\frac{n}{T^{c_0}\log n}\Big).
    \nonumber
\end{align}
Hence, we can proceed on the event $k^*\ge T- \lceil\frac{n}{\log n}\rceil$ by ruling out the probability of $\exp(-\frac{\beta^2n}{4\log n}) + \exp(-\frac{\beta}{2}\frac{n}{T^{c_0}\log n})$. By \cref{lem: corrupted error}, the error of $\bX_{k^*+1}$ is lower bounded by $\|\bX_{k^*+1}-\bx^*\|^2\ge \min_{i\in B}|\epsilon_i|^2$. 

We further show that the subsequent steps cannot reduce such error too much. 
We set 
$\kappa = \frac{100C}{n}$
in \cref{lem: fast-contraction} to yield that for any $k$, if  $\bX_{k+1}$ is obtained from $\bx_{k}$ by projecting onto an uncorrupted row, then the error satisfies
$\|\bX_{k+1}-\bx^*\|^2 \geq 0.99 \|\bx_k-\bx^*\|^2$ 
with probability at least 
$1 - \frac{100C}{n}$
for some absolute constant $C$. Starting from $\bX_{k^*+1}$, we will run no more than $T-k^*-1\le \lfloor\frac{n}{\log n}\rfloor$ iterations to obtain $\bX_T$. By the definition of $k^*$, none of these iterations performs a projection onto a corrupted row. Hence, these iterations either do not change the iterate or perform projection onto an uncorrupted row. We need only consider the projections onto uncorrupted rows, and the number of such iterations is bounded by $\lfloor\frac{n}{\log n}\rfloor$.  By a union bound, with probability at least 
$1- \frac{100C}{n}\cdot \lfloor\frac{n}{\log n}\rfloor\ge 1-\frac{100C}{\log n}$
, each of these projections (onto some uncorrupted row) cannot reduce the squared $\ell_2$ error by a factor smaller than  $0.99$. 
Overall, we will have 
\begin{align}
    \| \bX_T-\bx^*\|^2 \ge \Big(0.99\Big)^{\lfloor\frac{n}{\log n}\rfloor}\|\bX_{k^*+1}-\target\|^2\ge \Big(0.99\Big)^{\lfloor\frac{n}{\log n}\rfloor}\min_{i\in B}|\epsilon_i|^2,
     \nonumber
\end{align}
with probability at least $1 - \exp(-\frac{\beta^2n}{4\log n}) - \exp(-\frac{\beta}{2}\frac{n}{T^{c_0}\log n})-\frac{100C}{\log n}$. 
\end{proof}



\section{Numerical Simulations}\label{sec:numerics}

The experiments presented in this section were conducted using MATLAB R2022b on an iMac (2023) featuring an Apple M3 chip (8-core CPU) and 24 GB of unified memory. In our simulations, we generate the rows of the matrix $\bm{A} \in \mathbb{R}^{m \times n}$ with $m = 50000$ and $n = 100$ as i.i.d. samples uniformly distributed on the unit sphere $\mathbb{S}^{n-1}$. The target vector $\bm{x}^* \in \mathbb{R}^n$ is also sampled uniformly from $\mathbb{S}^{n-1}$, while the initial iterate $\bm{x}_0$ is set to the zero vector. For a given corruption level $\beta \in (0,1)$, the corruption vector $\bm{\epsilon}$ is generated to corrupt the first $\beta m$ measurements and its nonzero entries are i.i.d. sampled uniformly from the interval $[-5,5]$. 
All experiments are run for $200n = 20000$ iterations with $q = 0.5$. The error $\|\bX_{k} - \target\|$ is recorded at each iteration $k$ and averaged over 10 independent trials.  The MATLAB codes for reproducing our simulations are available online.\footnote{\href{https://github.com/wtree101/matlab-rk-analysis}{\texttt{https://github.com/wtree101/matlab-rk-analysis}}}

\subsection{Convergence under $D=\log(T),~\alpha m,~m$}
    We fix $\beta = 0.01$ and compare the {approximation} error versus iteration and runtime for subsample sizes $D \in \{4, 40,5000, 50000\}$ in \Cref{fig:qrk_error_vs_iter_and_time}. Here, $D\in\{4,40\}$ simulates the logarithmically small subsample size in our theory, while $D=5000$ simulates the $\Theta(m)$ subsampled size in \cite{haddock_subsampled_2023}, and $D=50000$ simulates the full-sample quantiles in \cite{haddock_quantilebased_2022}. 
    Across all $D$, the contraction rates per iteration are similar, but the computational cost per iteration grows with $D$, and hence regarding the runtime QRK with $D=5000,50000$ is much slower. 
    While our \cref{algo:quantileRK} draws the quantile subsample with replacement (as in \cite{haddock_quantilebased_2022}), we also tested QRK that draws the subsample without replacement (as adopted by \cite{haddock_subsampled_2023}). We find that both versions have  similar numerical performance; see \Cref{fig:qrk_error_vs_iter_and_time} (a) and the associated text.
         \begin{figure}[htbp]
                                \centering \begin{tabular}{cc}
                                                    \includegraphics[width=0.45\textwidth]{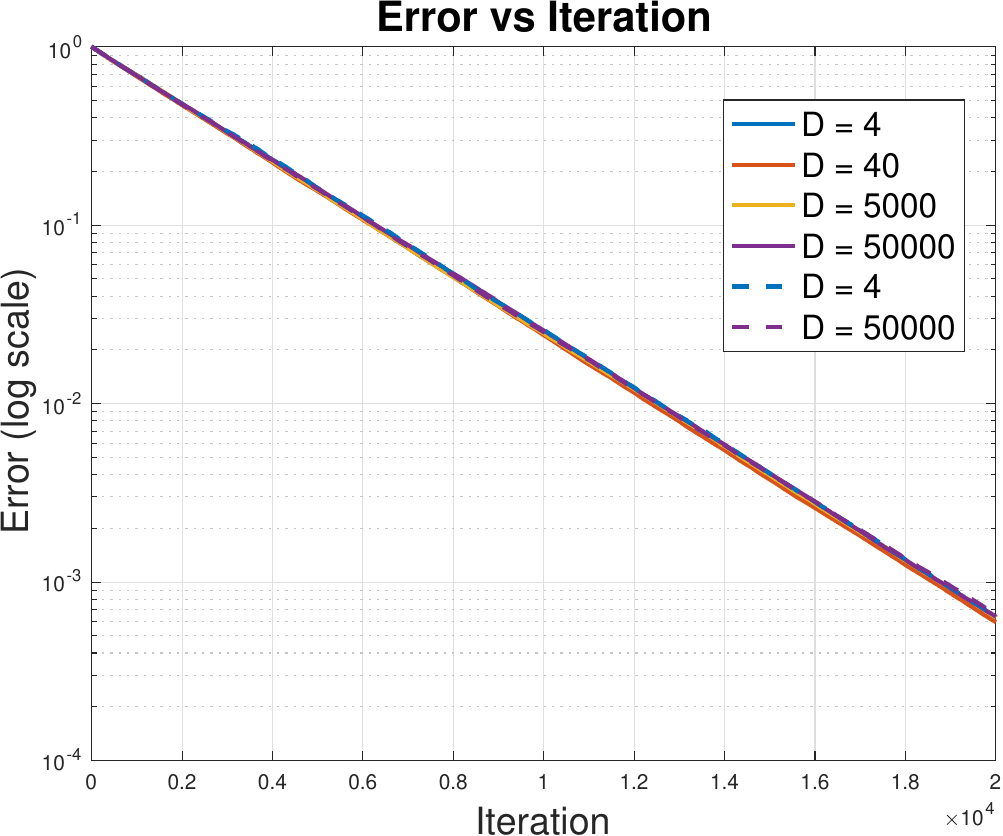} &
                                                    \includegraphics[width=0.45\textwidth]{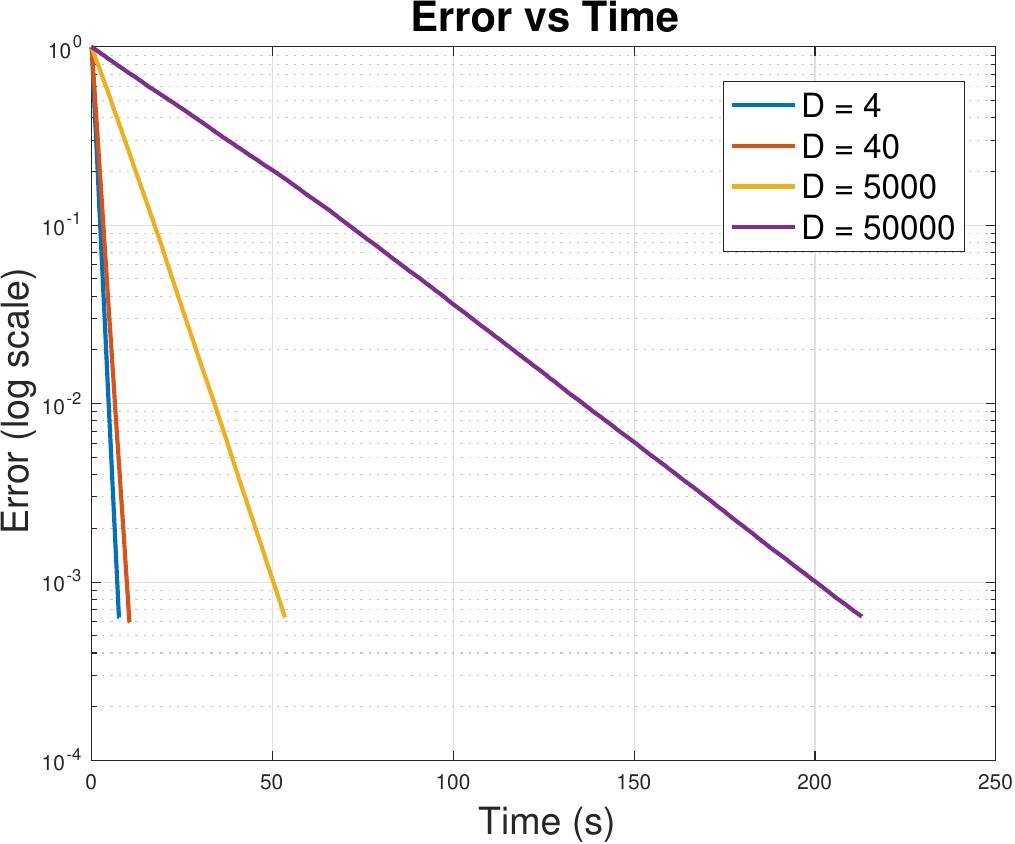} \\
                                                    \small (a) Error vs. iteration & \small (b) Error vs. runtime \\
                                                \end{tabular}
                                                \caption{Empirical convergence of QRK for $\beta=0.01$ and $q=0.5$ with varying $D$. (a) Error vs. iteration (solid lines: subsampling with replacement; dashed lines: without replacement); (b) Error vs. runtime.}
                                                \label{fig:qrk_error_vs_iter_and_time}
                \end{figure}

    \subsection{Dependence of $D$ on $(T,\beta)$}
    We numerically simulate the impact of $T,\beta$ on the subsample size. To this end,
    we plot the ``error versus iteration'' curves for varying   $(\beta,D)\in \{0.01, 0.06, 0.11\}\times  \{4, 8, 12\}$. 
    When $\bx_{k+1}$ is projected onto a corrupted row, the error can significantly increase and this leads to ``jumps'' in the curves.
     While we use corruption bounded by $5$ in our experiments, in general, such jumps can ruin the entire QRK procedure when the corruption is set arbitrarily large   (see \cref{lem: corrupted error}). 
     Therefore, {\it one has to avoid the occurrence of such jumps over the desired number of iterations.} In light of this, our experimental results in \Cref{fig:qrk_beta_effect} lead to the following conclusions that are consistent with our theory:
    \begin{itemize}
        \item  Under a fixed $T=20000$,  we compare \Cref{fig:qrk_beta_effect} (a)--(c) and find that increasing the corruption level $\beta$ necessitates a larger subsample size $D$ to avoid the occurrence of a  jump in error before the completion of $T$ iterations. For instance, under $D=8$, 
        the error decreases over the first $T$ iterations when $\beta = 0.06$, while a jump occurs when $\beta = 0.11$, and we have to use a larger $D$ (such as $D=12$) to avoid the jump. This corroborates the $\frac{1}{\log(1/\beta)}$ dependence of $D$ on $\beta$. 
        \item   Under a fixed corruption level $\beta$, we compare the curves within \Cref{fig:qrk_beta_effect} (b) or (c) and find that increasing $T$ will require a larger $D$. Specifically, in \Cref{fig:qrk_beta_effect} (c) with $\beta =0.11$, the error curve of $D=8$ decreases over the first $T_1=10000$ and then encounters the first jump, while the curve of $D=12$ decreases over the $T=20000$ iterations. This is consistent with the $\log(T)$ dependence of $D$ on $T$. 
                    \end{itemize}

    \begin{figure}[htbp]
        \centering
        \begin{tabular}{ccc}
            \includegraphics[width=0.3\textwidth]{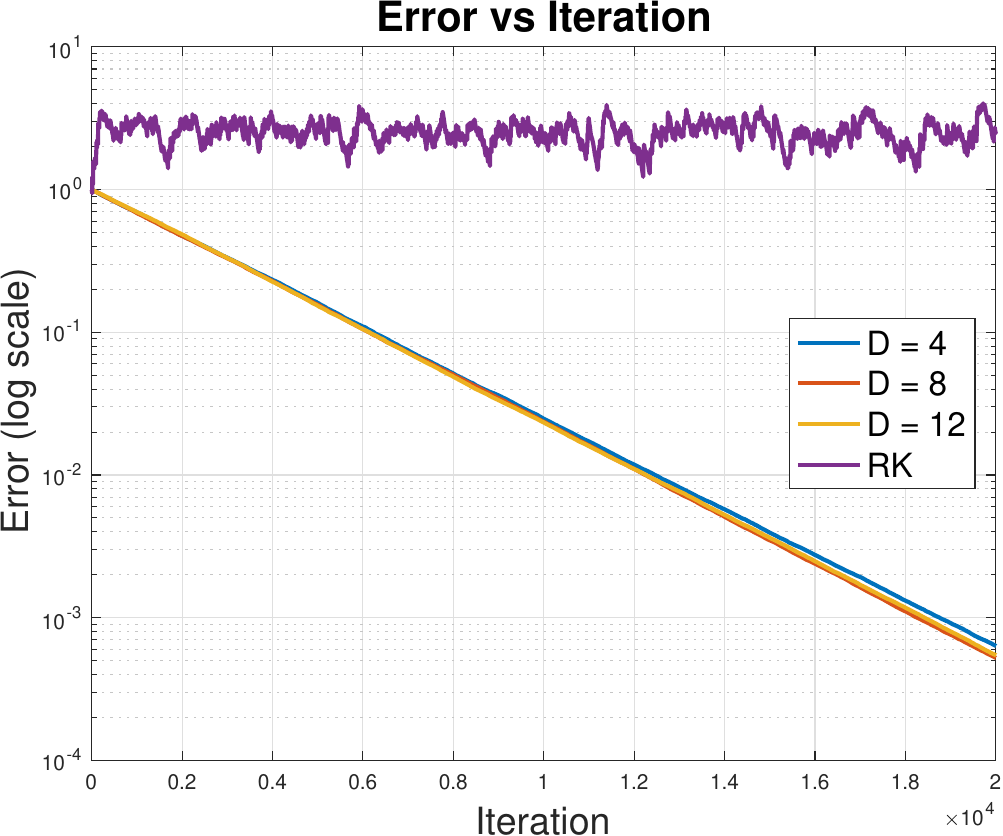} &
            \includegraphics[width=0.3\textwidth]{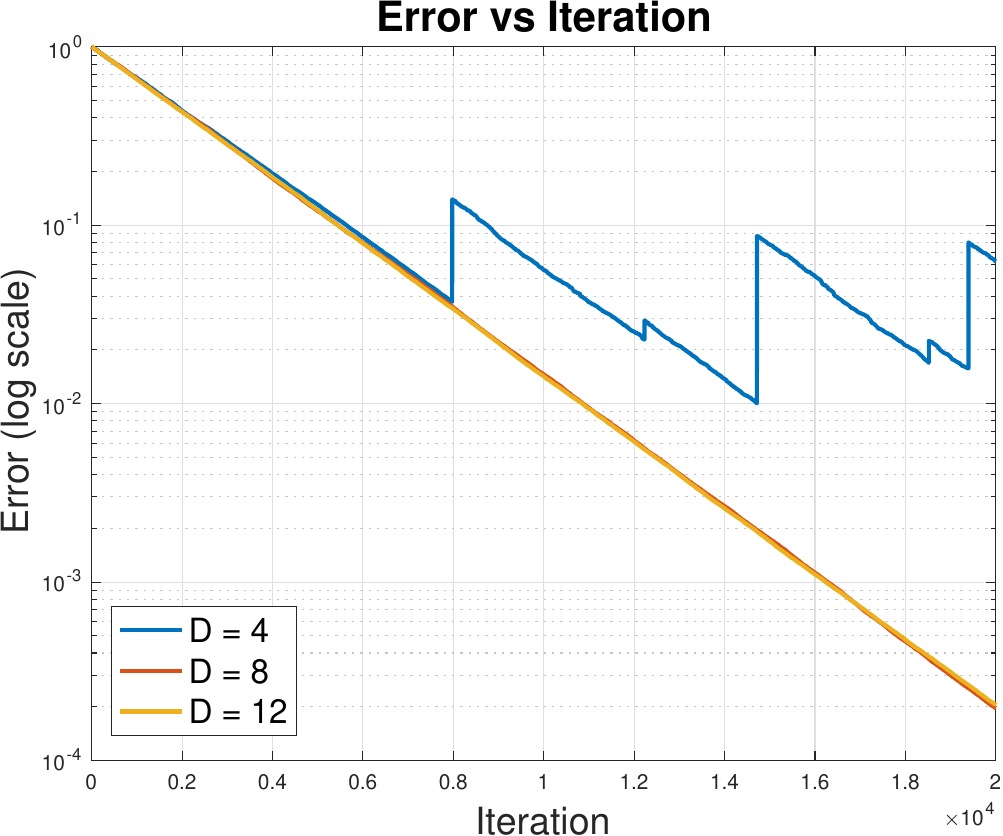} &
            \includegraphics[width=0.3\textwidth]{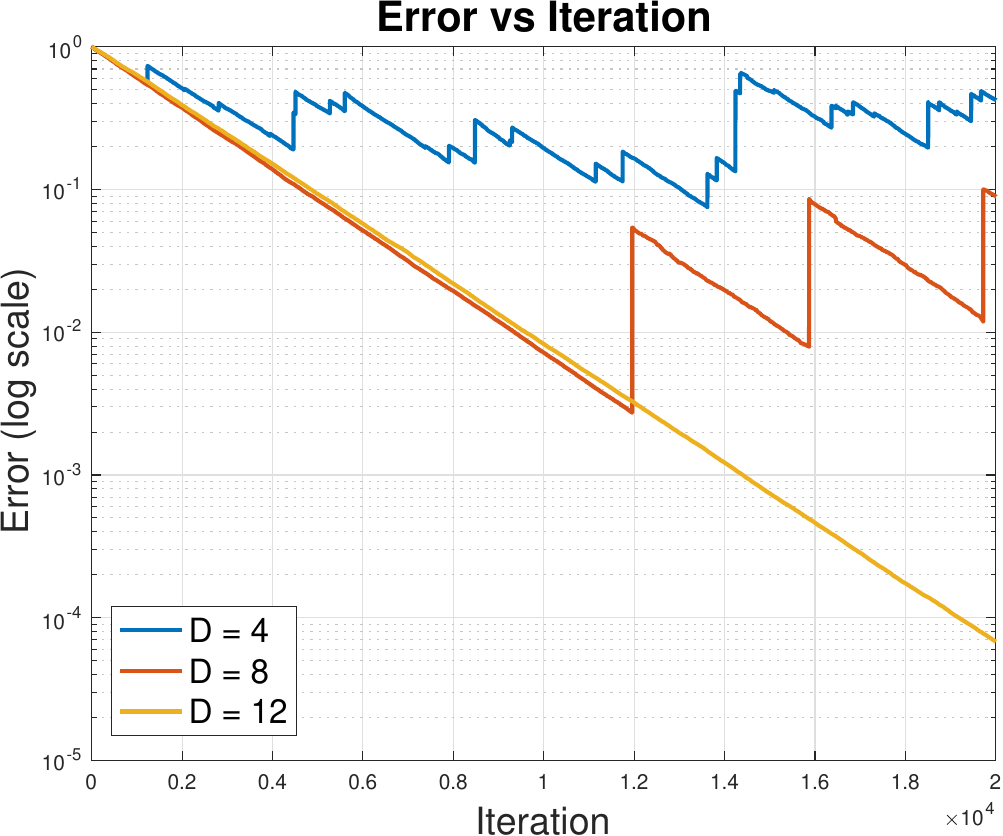} \\
            \small (a) $\beta=0.01$ & \small (b) $\beta=0.06$ & \small (c) $\beta=0.11$
        \end{tabular}
        \caption{Error vs. iteration for QRK under different corruption levels $\beta$ and subsample sizes $D$.} 
        \label{fig:qrk_beta_effect}
    \end{figure}

\section{Concluding Remarks} \label{sec:conclusion}
 
In this work, we provided a theoretical analysis on the subsample size of the quantile-based randomized Kaczmarz (QRK) method. Our main contribution is to identify (the order of) the minimal subsample size $D$ required for QRK to converge over $T$ iterations. {In particular, we show} that $D = \Theta\big(\frac{\log T}{\log(1/\beta)}\big)$ is both sufficient and necessary, and note that such a subsample size is typically a massive reduction on the previously known $D=\Theta(m)$. Our work thus bridges the gap between the previous theoretical guarantees for QRK, which require full-sample quantiles, and its practical implementations, which only use a very small subsample to compute the quantiles. Numerical experiments corroborate our theoretical findings and confirm that a fairly small subsample size is sufficient for the applications of QRK to solving large corrupted linear systems.

In order to prove our results, we introduced a number of new technical ingredients. These include two-sided probabilistic bounds on the subsampled quantiles, a stopping time to control $\mathbb{E}\big(\|\bX_T-\bx^*\|^2 1_{\Omega^c}\big)$, and a probabilistic statement on the performance limit of standard RK. We believe that these techniques are of interest to the analysis of (Q)RK and other randomized algorithms. In fact, since its introduction     \cite{haddock_quantilebased_2022}, QRK has attracted a lot of subsequent research interest such as block QRK \cite{cheng2022block,needell2014paved}, reverse  QRK \cite{battaglia2025reverse},  subspace constrained QRK \cite{lok2024subspace}, sparse QRK \cite{zhang2025quantile}, QRK for corrupted tensor linear system \cite{castillo2025quantile,ma2022randomized}, QRK under time-varying noise and corruption \cite{coria2024quantile}, among others. Note that all of these works require quantiles computed from the full sample, and we believe our techniques can be similarly used to substantially reduce the subsample size in these QRK variants. We leave these promising directions for future work.

 In addition to these future directions, we note that while the order of the minimal subsample size was found in our work, the multiplicative constants remain unspecified. Another research direction is to derive explicit constants for the upper bound and lower bound on $D$. Particularly, for a  fixed small enough $\beta>0$, it is of great interest to find explicit constants $c_*<C^*$ which are ideally close, such that
\begin{itemize}
    \item QRK with $D\ge C^*\log T$ converges over the first $T$ iterations, with high probability;
    \item QRK with $D\le c_*\log T$ returns a bad estimate $\bx_T$, with high probability.
\end{itemize}
        A further question will be to explore a potential phase transition phenomenon in QRK, that is, whether there exists a pair $(c_*,C^*)$ satisfying $\frac{C^*}{c_*}=1+o(1)$ such that the above two bullet points hold.

\bibliographystyle{siamplain}
\bibliography{references}
\end{document}